\documentclass[
twocolumn, 
]{IEEEtran}
\usepackage[T1]{fontenc}
\usepackage[latin1]{inputenc}
\usepackage[english]{babel} 
\usepackage{amssymb, amsmath, amsthm, amsfonts}
\usepackage{placeins}
\usepackage{pdfsync}

\allowdisplaybreaks

\RequirePackage{enumerate}
\RequirePackage[colorlinks, citecolor=black, urlcolor=black, linkcolor=black, pdffitwindow=false, pdfstartview={FitH}]{hyperref}

\usepackage{natbib}
\makeatletter

\newtheorem{thm}{Theorem} 
\newtheorem{cor}{Corollary} 
 
\newtheorem{prop}{Proposition}

\newtheorem{dfn}{Definition}
\newtheorem{rem}{Remark}
\usepackage{verbatim}
\usepackage{setspace}

\usepackage{bbm}
\usepackage{amsbsy}
\usepackage{babel}
\usepackage{color}

\newcommand{\mathe}{\mathrm{e}}

\newcommand{\EXP}{\ensuremath{\mathbb{E}}}
\newcommand{\IND}{\mathbb{I}}
\newcommand{\KT}{\textsc{kt}}
\newcommand{\Np}{\ensuremath{\mathbb{N}_+}}
\newcommand{\N}{\ensuremath{\mathbb{N}}}
\newcommand{\PROB}{\ensuremath{\mathbb{P}}}

\newcommand{\M}{\ensuremath{\mathfrak M}}

\makeatother

\begin{document}
\title{ 
   About Adaptive Coding on Countable  Alphabets\IEEEauthorrefmark{4}\thanks{\IEEEauthorrefmark{4}The material
     in this paper was presented in part at the 
$23^{\text{rd}}$ Intern. Meeting on Probabilistic, Combinatorial, and Asymptotic Methods for the Analysis of Algorithms (AofA'12), Montr\'eal, Qu\'ebec,
Canada, June 2012}}

\author{
\IEEEauthorblockN{Dominique Bontemps\IEEEauthorrefmark{3}\thanks{\IEEEauthorrefmark{3}supported by Institut de Math\'ematiques de Toulouse, Universit\'e de Toulouse} \and  St\'ephane Boucheron\IEEEauthorrefmark{1}\thanks{\IEEEauthorrefmark{1}supported by Network of Excellence \textsc{pascal ii},  Laboratoire de Probabi\-lit\'es et Mod\`eles Al\'eatoires, Universit\'e Paris-Diderot} \and Elisabeth Gassiat\IEEEauthorrefmark{2}\thanks{\IEEEauthorrefmark{2}supported by Network of Excellence \textsc{pascal ii}, Laboratoire de Math\'ematiques d'Orsay, Universit\'e Paris-Sud}
}
}



\maketitle

\begin{abstract}
This paper sheds light on  adaptive  coding with respect to 
classes of memoryless sources over a countable alphabet defined by an envelope function  
with  finite and non-decreasing hazard rate (log-concave envelope distributions). 
We prove that the auto-censuring \textsc{(ac)} code introduced by \citet{Bon11} is adaptive with respect to 
the collection of such classes. The analysis builds on the tight characterization 
of universal redundancy rate in terms of metric entropy   
by \citet{MR1604481}
and on a careful analysis of the performance of the \textsc{ac}-coding algorithm. The latter relies on non-asymptotic  bounds 
for maxima of samples from discrete distributions with finite and non-decreasing hazard rate. 
\end{abstract}

\bibliographystyle{IEEEtranSN} 
\section{Introduction}
\label{sec:introduction}

\subsection{From universal coding to adaptive coding}
\label{sec:universal}

This paper is concerned with   problems of \emph{adaptive} that is
\emph{twice universal} or \emph{hierarchical universal} coding over a countable
 alphabet $\mathcal{X}$ 
(say the set of positive integers $\Np$ or
the set of integers $\N$). 
Sources over  alphabet $\mathcal{X}$ 
are probability distributions on the set $\mathcal{X}^\N$ of infinite sequences of symbols from $\mathcal{X}$. 
In this paper, the symbol $\Lambda$ will be used to  denote various collections of sources on  alphabet $\mathcal{X}.$ 
The symbols emitted by a source are denoted by a sequence $\mathbf{X}$ of 
$\mathcal{X}$-valued random
variable $\mathbf{X}=\left(X_n\right)_{n\in\N}. $ 
If $\PROB$  is the distribution of $\mathbf{X},$ $\PROB^n$ denotes the
distribution of the first $n$ symbols $X_{1:n}=(X_1,...,X_n),$ and we 
let $\Lambda^n=\{\PROB^n : \PROB \in\Lambda\}$.

Throughout the paper, we will rely on the correspondence between
non-ambiguous codes and probability distributions and refer to codes through coding
probabilities \citep[see][for a gentle introduction to the notion of coding probability]{cover:thomas:1991}. 
The \emph{expected redundancy} of any (coding) distribution $Q^n\in
 {\mathfrak{M}}_1({\mathcal{X}}^n)$ with respect to $\PROB$ is equal to
the  Kullback-Leibler divergence (or relative entropy)  between
$\PROB^n$  and $Q^n$, 
\begin{eqnarray*}
D(\PROB^n,
 Q^n) &=& \sum_{\mathbf{x}\in {\cal X}^n} \PROB^n\{\mathbf{x}\} 
 \log \frac{\PROB^n(\mathbf{x})}{Q^n(\mathbf{x})} \\
& = &\EXP_{\PROB^n}\left[\log 
 \frac{\PROB^n(X_{1:n})}{Q^n(X_{1:n})}\right] \, .
\end{eqnarray*} Up to a constant, the expected redundancy of $Q^n$ with respect to
$\PROB$ is the expected difference between the length of codewords
defined by encoding messages as if they were produced by $Q^n$  and
the ideal codeword length when encoding messages produced by
$\PROB^n.$ In the language of mathematical statistics, it is also the
\emph{cumulative entropy risk} suffered by estimating $\PROB^n$ using $Q^n.$

 Notice that the definition of redundancy uses base $2$ logarithms. Throughout this text, $\log x$ denotes the base $2$ logarithm of $x$ while $\ln x$ denotes its natural logarithm.

Universal coding attempts to develop sequences of coding probabilities $(Q^n)_n$  so as to
minimize expected redundancy over a whole known class of sources. The
\emph{maximal redundancy}  of coding probability $Q^n$ with respect to
source class $\Lambda$ is
defined by
$$R^+(Q^n, \Lambda^n)=\sup_{\PROB\in\Lambda} D(\PROB^n, Q^n) \, . $$ 
The infimum  of $R^+(Q^n,\Lambda^n)$ is called the \emph{minimax
  redundancy} with respect to $\Lambda$,  
$$R^+(\Lambda^n)=\inf_{Q^n\in {\M}_1\left({\mathcal{X}}^n\right)}
R^+(Q^n,\Lambda^n).$$ 

As the design of almost minimax coding probabilities is usually not a
trivial task, looking for an apparently more ambitious goal,
\emph{adaptivity}, may seem preposterous. Indeed,  
whereas universality issues are ubiquitous in lossless coding theory \citep{csiszar:korner:1981,cover:thomas:1991},
and adaptivity has been  a central concept in Statistics
during the last two decades    \citep[See][and
references
therein]{MR1623559,barron:birge:massart:1999,donoho:johnstone:1994,donoho:Johnstone:Kerkyacharian:picard:1996,abramovich:benjamini:donoho:johnstone:2000,tsybakov:2004},
the very word  adaptivity barely made its way in the lexicon of
Information Theory. Nevertheless, adaptivity issues have been
addressed in
coding theory, sometimes using different expressions to name things. Adaptive coding is sometimes called twice universal
coding \citep{Rya1984,MR1101099,MR1895085,MR2451044} or hierarchical
universal coding  \citep{MeFe98}. We pursue this endeavour. 

A sequence $(Q^n)_{n}$  of coding probabilities   is said
to be \emph{asymptotically adaptive} with respect to a collection
$(\Lambda_m)_{m\in \mathcal{M}}$ of  source classes if for all $m\in
\mathcal{M}$,  
\begin{eqnarray*}
R^+(Q^n, \Lambda_m^n)&= &\sup_{\PROB \in \Lambda_m} D(\PROB^n,Q^n) \leq (1+o_m(1))R^+(\Lambda_m^n) 
\end{eqnarray*}
as $n$ tends to infinity. In words, a sequence of  coding
probabilities is adaptive with respect to a collection of source
classes if it asymptotically achieves minimax redundancy over all
classes. Note that this is a kind of first order requirement, the
$o_m(1)$ term may tend to $0$ at a rate that depends on the source
class $\Lambda_m$. 

This is not the only way of defining
adaptive compression, more stringent definitions are possible
\citep[See][Section 1.5]{catoni:2004}. This last reference 
describes oracle inequalities for the context-tree weighting method
\citep{Wil98},
 a successful attempt to achieve adaptivity with respect to an infinite collection of
source classes on a finite alphabet indexed by their memory structure \citep{catoni:2004}.

The present paper 
describes another successful attempt to achieve adaptivity with
respect to an infinite collection of source classes on a countable alphabet.

\subsection{Adaptive coding with respect to a collection of envelope classes}
\label{sec:envelope}

Pioneering results by \cite*{MR514346}, \citeauthor*{gyorfi1993uns} (\citeyear{gyorfi1993uns,MR1281931}) show  
that, as soon as the alphabet is infinite,  finite minimax redundancy,
that is the possibility of achieving universality,  is not a trivial property even for classes of memoryless sources. 
\begin{prop} \label{prop:kieffer}
If a class $\Lambda$ of stationary sources over a countable alphabet ${\mathcal{X}}$
has finite \emph{minimax redundancy}
then  there exists a probability distribution 
 $Q$ over $\mathcal{X}$ such that for every
$\PROB\in\Lambda$ with $\lim_n H(\PROB^n)/n <\infty$ where $H(\PROB^n)=\sum_{\mathbf{x} \in \mathcal{X} } -\PROB^n(\mathbf{x})\log \PROB^n(\mathbf{x})$
(finite Shannon entropy rate), $Q$ 
satisfies   $D(\PROB^1,Q)<\infty .$
\end{prop}
This observation contrasts with what we know about the finite alphabet setting where 
coding probabilities asymptotically achieving minimax redundancies have been described
\citep{xie:barron:2000,barron:rissanen:yu:1998,yang:barron:1998,xie:barron:1997,barron:clarke:1994}. This
even contrasts with recent  delicate asymptotic results for coding over large finite alphabets with unknown size \citep{SzpWei10,YanBar13}. 

This  prompted  \cite*{boucheron:garivier:gassiat:2006} to investigate
the redundancy of specific memoryless source classes, namely 
classes defined by an \emph{envelope} function. 

\begin{dfn}\label{dfn:envelope:class}
Let  $f$ be  a mapping from  $\Np$ to $[0,1],$ with $1\leq \sum_{j>0} f(j)<\infty$.  
The \emph{envelope class} $\Lambda_f$ defined by the function $f$   is 
the collection of stationary  memoryless sources with first marginal
distribution dominated by $f$,  
\begin{eqnarray*}
\Lambda_f 
& =&\Bigl\{ \PROB~:~~\forall x\in {\Np},\;\PROB^1\{x\}\leq
f(x)~, \Bigr. 
\\
& & \Bigl. \text{ and } \PROB \text{ is stationary and memoryless.}
\Bigr\}\, . 
\end{eqnarray*}
\end{dfn}
An envelope function defines an \emph{envelope distribution}.  The minimax
redundancy  
of the source classes we are interested in is, up to the first order, asymptotically
determined by  the tail behavior of the envelope distribution. 
\begin{dfn}\label{def:env:distribution} Let $f$ be an  envelope function.
The associated \emph{envelope distribution} has lower endpoint 
 $l_f = \max \{k\colon \sum_{j\geq k} f(j)\geq 1\}$. The envelope distribution $F$
 is  defined by 
$F(k)=0$ for $k<l_f,$ and $F(k)= 1- \sum_{j>k} f(j)$ for $k\geq l_f$. 
The tail function $\overline{F}$ is defined by $\overline{F}=1-F$.   
The associated probability mass function coincides with $f$ for  $u>l_f$ and 
is equal to  $F(l_f)\leq f(l_f)$ at $u=l_f$. 
\end{dfn}
This envelope probability distribution 
plays a special role  in the analysis of   the minimax redundancy  $R^+(\Lambda^n_f)$.
\cite*{boucheron:garivier:gassiat:2006} related the summability of the envelope function and the minimax redundancy of the envelope class.
They  proved almost matching 
upper  and lower bounds on minimax redundancy
for  envelope classes. The next theorem provides an upper-bound on the minimax redundancy
of envelope classes and suggests  general design principles for
universal coding over envelope classes and for adaptive coding over a
collection of envelope classes. 
\begin{thm}  \citep*{boucheron:garivier:gassiat:2006}
\label{prop:upperbound}
If $\Lambda$ is an envelope  class of memoryless sources, 
with  the tail envelope function
$\overline{F}$ 
then:
\begin{displaymath}
R^+(\Lambda^n) \leq \inf_{u : u\leq n} \, \left[ n
  \overline{F}(u)\log e + \frac{u-1}{2}\log n 
  \right] + 2\, .
\end{displaymath}
\end{thm}
If the envelope $F$ is known, if the message length $n$ is known, the
following strategy is natural: determine $u$ such that
$\overline{F}(u) \approx \frac{1}{n}$; choose a good universal coding
probability for  memoryless sources over alphabet $\{0,\ldots,u\}$;
escape symbols larger than $u$ using $0$ which does not belong to the
source alphabet;  encode the escaped sequence using the good universal
coding probability; encode all symbols larger than $u$ using a coding
probability tailored to the envelope distribution. If the upper bound
is tight, this strategy should achieve the minimax redundancy
rate. If the message length is not known in advance, using a doubling 
technique should allow to derive an online extension. As naive as this
approach may look, it has already proved fruitful. \citet{Bon11}
showed that 
the minimax redundancy of classes defined by exponentially vanishing
envelopes is half the upper bound obtained by choosing $u$ so as 
$\overline{F}(u) \approx 1/n$.

If we face a collection of possible envelope classes and the envelope
is not known in advance, we face two difficulties: there is no obvious
way to guess a reasonable threshold; once a threshold is chosen, there
is no obvious way to choose a coding probability for escaped symbols.


There are reasons to be optimistic. 
Almost adaptive coding techniques  for the collection of source classes defined 
by power law  envelopes were introduced in \citep*{boucheron:garivier:gassiat:2006}.  
Moreover, \citeauthor{Bon11} designed and analyzed the
\textsc{ac}-code (Auto-Censuring code) (described in Section
\ref{sec:ac-code}) and proved that this simple computationally
efficient online code is adaptive over the union of 
classes of sources  with exponentially decreasing envelopes (see
Definition \ref{dfn:subexp}). 
As the \textsc{ac}-code does not benefit from any side information concerning the envelope, it is natural to ask whether it is adaptive to a larger class of sources. That kind of question has been addressed in data compression by \citet{Gar06} who proved that Context-Tree-Weighting \citep{Wil98,catoni:2004}  is 
adaptive over renewal sources \citep{csiszar:shields:1996} while it
had been designed to compress sources with bounded memory. 
In a broader context, investigating the situations where an appealing procedure is minimax motivates the \emph{maxiset} approach pioneered in \citep{MR1848302,KePi02}.   

\subsection{Roadmap}
\label{sec:answers}

This paper shows that the \textsc{ac}-code is adaptive over the
collection of envelope classes that lie between the exponential envelope classes investigated
 in \citep{boucheron:garivier:gassiat:2006,Bon11} 
and the classes of sources with finite alphabets (Theorem
\ref{thm:main-result-new}). The relevant envelopes are characterized
by the fact that they have non-decreasing hazard rate (see Section
\ref{sec:hazard}). This distributional property implies that the
corresponding envelope
distributions fit nicely in the framework of extreme value theory (See
Section \ref{sec:slow-variation}), smoothed version of the envelope
distribution belong to the so-called Gumbel domain of attraction, and
this implies strong concentration properties for maxima of i.i.d. samples
distributed according to envelope distributions. 
As the \textsc{ac}-code uses
 mixture coding over the \emph{observed alphabet} in a sequential way,
the  intuition provided by Theorem \ref{prop:upperbound} suggests that the \textsc{ac}-code should perform well when
 the largest symbol in a message of length $n$ is close to the
 quantile of order $1-1/n$ of the envelope distribution. This
 concentration property is a consequence of the non-decreasing hazard
 rate assumption \citep{boutho2012}. Moreover we check in the Appendix (see Section \ref{sec:numb-dist-symb})
that if the sampling distribution has the non-decreasing hazard rate
property, on average, the size of the largest symbol and the number of distinct
symbols in the sample differ by a constant.

The non-decreasing hazard  rate assumption has far reaching
implications concerning the slow variation property of the quantile
function of the envelope distribution (Section
\ref{sec:slow-variation}) that prove instrumental in the derivation of
matching lower bounds for the minimax redundancy of the corresponding
envelope classes. In Section \ref{sec:minimax-redundancies}, we
revisit the powerful results concerning extensions of minimax
redundancy  by \citet{MR1604481}. Advanced results from regular
variation theory shed new light on the small classes where the lower
bounds from \citep{MR1604481} are known to be tight 

In words, borrowing ideas from extreme value theory \citep{FalHusRei11,HaaFei06,MR2108013,Res87}, we prove that 
if the envelope distribution function has finite and non decreasing hazard rate  (defined in Section~\ref{sec:hazard}): 
i) an explicit formula connects  the minimax redundancy and the envelope distribution;
ii)  the \textsc{ac}-code asymptotically achieves the minimax redundancy, that is the \textsc{ac}-code
 is adaptive with respect to the collection of envelope classes with finite and non decreasing hazard rate.

The paper is organized as follows. Section~\ref{sec:ac-code} describes
the  \textsc{ac}-code. 
Section~\ref{sec:hazard} provides  notation and definitions concerning hazard rates.
The main result concerning the adaptivity of 
the \textsc{ac}-code over classes with  envelopes with finite and non-decreasing hazard rate 
is stated in Section~\ref{sec:main-result}. The minimax redundancy of source classes defined by envelopes with finite and non-decreasing hazard rate 
 is characterized in Section~\ref{sec:minimax-redundancies}. Section~\ref{sec:redundancy-ac-code} is dedicated 
to the characterization of the redundancy of the \textsc{ac}-code over source classes  defined by envelopes with finite and non-decreasing hazard rate. 
 
\section{The AC-code}
\label{sec:ac-code}

The \textsc{ac}-code encodes a sequence $x_{1:n}= x_1,\ldots,x_n$ of symbols from $\mathbb{N}_+= \mathbb{N}\setminus \{0\}$ 
in the following way. For $i \colon 1\leq i\leq n,$
let $m_i=\max_{1\leq j\leq i} x_j.$ The $i^{\text{th}}$ symbol is a \emph{record} if $m_i\neq m_{i-1}.$ Let $n^0_i$  be the number of records up to index 
$i$. The $j^{\text{th}}$  record is denoted by $\widetilde{m}_j$. From the definitions, $\widetilde{m}_{n^0_i}= m_i$ for all $i.$
Let $\widetilde{m}_0=0$ and let $\widetilde{\mathbf{m}}$ be derived
from the sequence of differences between records and  terminated by a
$1$, 
$\widetilde{\mathbf{m}}=(\widetilde{m}_{i}
-\widetilde{m}_{i-1}+1)_{1\leq i\leq n^0_n}1$. The last $1$ in the sequence serves as a terminating symbol.
 The symbols in  $\widetilde{\mathbf{m}}$ are encoded 
using Elias penultimate code \citep{MR0373753}. This sequence of codewords forms $C_E$.  The sequence of censored symbols $\widetilde{x}_{1:n}$  is defined by
$\widetilde{x}_i = x_i  \IND_{x_i\leq m_{i-1}}$. The binary string
$C_M$ is obtained by arithmetic encoding of  $\widetilde{x}_{1:n}0.$ 
\begin{rem} Let  $x_{1:n} \in \Np^n$  be 
$$
\emph{5} \, \emph{15} \,\,  8 \,  1 \, \emph{30} \, \,  7 \,  1  \,  2  \,  1  \,  8 \,  4 \,   7 \,  15 \,\,   1  \,  5 \,  17 \,  13\, \,
4 \, 12 \,\, 12 \, ,
$$ (records are italicized) then 
$m_{1:n}$ is 
$$
5\, 15\, 15\, 15\, 30\, 30\, 30\, 30\, 30\, 30\, 30\, 30\, 30\, 30\, 30\, 30\, 30\, 30\, 30\, 30
$$
 and $\tilde{x}_{1:n}0$  is parsed into $4$ substrings terminated by $\emph{0}$
$$
\underbrace{\emph{0}}_{i}\, \underbrace{\emph{0}}_{ii}\, \, \underbrace{8\, \, 1\, \, \emph{0}}_{iii}\, \, \underbrace{7\, \, 1\, \, 2\, \, 1\, \, 8\, \, 4\, \, 7\, 15\, \, 1\, \, 5\, 17\, 13\,\,
4\, 12\,\, 12\,\,  \emph{0}}_{iv}
$$
while $\tilde{m}$  is 
$
6\,  11\,\,  16\,  \, \emph{1}
$.
\end{rem}

The coding probability used to (arithmetically) encode $\widetilde{x}_{1:n}0$ is
\begin{displaymath}
  Q^{n+1} (\widetilde{x}_{1:n}0) = Q_{n+1}(0\mid x_{1:n}) \prod_{i=0}^{n-1} Q_{i+1} (\widetilde{x}_{i+1}\mid x_{1:i}) \, . 
\end{displaymath}
with 
\begin{displaymath}
  Q_{i+1} \left(\widetilde{X}_{i+1} = j \mid X_{1:i}=x_{1:i} \right) =  \frac{n^j_i+\tfrac{1}{2}}{i+\tfrac{m_i+1}{2}}
\end{displaymath}
where $n_i^j$ is the number of occurrences of symbol $j$ amongst the first $i$ symbols (in $x_{1:i}$). We agree on $n^j_0=0$ 
for all $j>0$. If $i<n,$ the event $\{ \widetilde{X}_{i+1}=0\}=\{X_{i+1}=M_{i+1}>M_i\}$ has conditional probability 
$Q_{i+1} \bigl(\widetilde{X}_{i+1} = 0\mid X_{1:i}=x_{1:i} \bigr) =  \tfrac{{1}/{2}}{i+{(m_i+1)}/{2}}.$ 
Note that $0$ is always encoded as a new symbol: if $x_{i+1}=j>m_i$, the \textsc{ac}-code encodes a $0$, but $n_i^j$ rather than $n_i^0$ is incremented. 

In words, the mixture code consists of progressively enlarging the alphabet and feeding an arithmetic coder  with Krichevsky-Trofimov mixtures 
over the smallest alphabet seen so far \citep[See][for a gentle
introduction to Krichevsky-Trofimov
mixtures]{cesa-bianchi:lugosi:2006}. 

\citet{Bon11} describes a nice simple way of  interleaving the Elias
codewords and the mixture code in order to perform \emph{online}
encoding and decoding. Substrings of $\tilde{x}_{1:n}0$ terminated by
$0$ are fed online to an arithmetic coder over the relevant alphabet
using properly adapted Krichevsky-Trofimov mixtures, after each $0$, 
the corresponding symbol from $\tilde{m}$ is encoded using the
self-delimited Elias code and transmitted. The alphabet used to encode
the next substring of    $\tilde{x}_{1:n}0$ is enlarged and the
procedure is iterated. The last symbol of $\tilde{m}$ is~a~$1$, it
signals the end of the message.


 \section{Hazard rate and envelope distribution}
 \label{sec:hazard}

The envelope distributions we consider in this paper are characterized  
by the behavior of their hazard function $n \mapsto -\ln
\overline{F}(n)$. The probabilistic analysis of the performance of the
\textsc{ac}-code  borrows tools and ideas from extreme value
theory. As the theory of extremes for light-tailed discrete random
variables is plagued by interesting but distracting paradoxes, 
following \citet{And70}, it proves convenient to define a continuous distribution function starting 
from the envelope distribution function $F$. This continuous
distribution function will be called the \emph{smoothed envelope
distribution}, it coincides with the envelope distribution on
$\mathbb{N}$. Its hazard function is defined by linear interpolation:
the \emph{hazard rate},  that is the derivative of the hazard function is
well-defined on $\mathbb{R}_+ \setminus \mathbb{N}$ and it is piecewise
constant. The envelope distributions we consider here are such that
this hazard rate  is non-decreasing and finite.
The essential infimum of the hazard rate is~$b=-\ln \overline{F}(l_f)>0.$ 
Notice that the hazard rate is finite on $ [l_f -1,\infty)$ if and only if $f$ has infinite support. 

We will also repeatedly deal with the quantile function of the
envelope distribution and even more often with the quantile function
of the smoothed envelope distribution. As the latter is continuous and
strictly increasing over its support, the quantile function of the
smooth envelope distribution is just the inverse function of the smooth
envelope distribution. The quantile function of the piecewise constant
envelope distribution is the left continuous generalized inverse:
\begin{displaymath}
  F^{-1}(p) =  \inf\{ k \colon k \in \mathbb{N},  F(k)\geq p\} \, . 
\end{displaymath}
If the hazard rate is finite, then $\lim_{p\to 1} F^{-1}(p) =\infty$. 
 Note that the smoothed envelope distribution has support $[l_f-1,
 \infty)$. 
Recall that if $X$  is distributed according to the smoothed envelope
distribution  $\lfloor X\rfloor +1$ and  $\lceil X\rceil$
are distributed according to the envelope distribution. 


\begin{rem}
  Assume that $F$  is a shifted geometric distribution: for some $l\in
  \Np$, some $q\in (0,1)$, for all $k\geq l$, $\overline{F}(k)
  =(1-q)^{k-l}$, so that the hazard function is $(k-l)\ln 1/(1-q)$ for
  $k\geq l$. The corresponding smooth distribution is the shifted
  exponential distribution with tail function $t\mapsto (1-q)^{t-l}$ for $t>l$. 
\end{rem}

The  envelopes introduced in the next definition provide examples where the associated continuous distribution 
function has non-decreasing hazard rate. Poisson distributions offer other examples. 
\begin{dfn}\label{dfn:subexp}
 The \emph{sub-exponential envelope
  class} with  parameters $\alpha\geq 1$ (shape),  $\beta>0$ (scale)  and $\gamma>1$   is the set $\Lambda
  (\alpha,\beta,\gamma)$ of probability mass functions $(p(k))_{k\geq 1}$   on the positive integers such that
$$
\forall k\geq 1,\; p(k)\leq f(k),\; \text{ where } f(k)=\gamma e^{ -\big(\tfrac{k}{\beta}\big)^{\alpha}}.
$$ 
\end{dfn}
Exponentially vanishing envelopes \citep{boucheron:garivier:gassiat:2006,Bon11}  are obtained by fixing $\alpha=1.$


\section{Main result}
\label{sec:main-result}

In the text, we will repeatedly use the next shorthand for the
quantile of order $1-1/t, t>1$ of the smoothed envelope
distribution. The function  
$U\colon[1,\infty) \to \mathbb{R}$ is defined by
\begin{equation}
U(t) =  F_s^{-1}(1-1/t)\label{eq:4}
\end{equation}
where $F_s$ is the smoothed envelope distribution. 
 
The main result may be phrased as follows. 
\begin{thm}\label{thm:main-result-new}
The \textsc{ac}-code is adaptive with respect to source classes defined by 
envelopes with finite and non-decreasing hazard rate. 

Let $Q^n$ be the coding probability associated with the \textsc{ac}-code, then  
if $f$ is an envelope with non-decreasing hazard rate,  and
$U\colon[1,\infty) \to \mathbb{R}$ is defined by \eqref{eq:4}, then 
\begin{displaymath}
R^+(Q^n; \Lambda^n_f) \leq (1+o_f(1)) (\log e) \int_1^n \frac{U(x)}{2x} \mathrm{d} x \, 
\end{displaymath}while
\begin{displaymath}
R^+(\Lambda_f^n) \geq  (1+o_f(1)) (\log e) \int_1^n \frac{U(x)}{2x} \mathrm{d} x \,   
\end{displaymath}
as $n$ tends to infinity.
\end{thm}
\begin{rem}
Note that the \textsc{ac}-code is almost trivially adaptive over
classes of memoryless sources with alphabet $\{1, \ldots, k\}, k \in
\Np$:  almost-surely eventually, the largest symbol in the sequence
coincides with the right-end of the source distribution, and the
minimaxity (up to the first order) of Krichevsky-Trofimov mixtures
settles the matter.
\end{rem}
The following corollary provides the bridge with \citeauthor{Bon11}'s
work on classes defined by  exponentially decreasing envelopes (See
Definition \ref{dfn:subexp}).
\begin{cor}\label{thm:main-result}
The \textsc{ac}-code is adaptive with respect to sub-exponential
envelope classes  $\cup_{\alpha\geq 1,\beta>0,\gamma>1} \Lambda(\alpha,\beta,\gamma)$.  
Let $Q^n$ be the coding probability associated with the \textsc{ac}-code, then  
\begin{eqnarray*}
R^+(Q^n; \Lambda^n(\alpha,\beta,\gamma))& \leq &
(1+o_{\alpha,\beta,\gamma}(1)) R^+(\Lambda^n(\alpha,\beta,\gamma))
\end{eqnarray*}
as $n$ tends to infinity.
\end{cor}

\citet{Bon11} showed that the \textsc{ac}-code is adaptive over exponentially decreasing envelopes, 
that is over $\cup_{\beta>0,\gamma>1} \Lambda(1,\beta,\gamma).$ Theorem~\ref{thm:main-result}  shows 
that the \textsc{ac}-code is  adaptive to both the scale and the shape parameter. 

The next equation helps in  understanding the relation between the redundancy of the \textsc{ac}-code and the metric entropy:
\begin{equation}\label{eq:bizarre}
  \int_1^t \frac{U(x)}{2x} \mathrm{d} x
=  \int_{0}^{U(t)} \frac{\ln (t \overline{F}_s(x))}{2} \mathrm{d}x \, .
\end{equation}
The elementary proof relies on the fact $t\mapsto U(\mathe^t)$
is the inverse of the hazard function of the smoothed envelope
distribution $-\ln \overline{F}_s$, it is given at the end of  the appendix. 
The left-hand-side of the equation appears (almost) naturally in the derivation of
the redundancy of the \textsc{ac}-code. The right-hand-side or rather an equivalent of it, appears during 
the computation of the minimax redundancy of the envelope classes considered in this paper.

The proof of Theorem~\ref{thm:main-result} is organized in two parts: Proposition~\ref{prop:minimax-redundancies} from Section \ref{sec:minimax-redundancies} 
gives a lower bound for  the minimax redundancy of source classes defined by envelopes with finite and non-decreasing hazard rate.

The redundancy of the \textsc{ac}-coding probability 
$Q^n$ with respect to $\PROB^n\in \Lambda^n_f$ is analyzed in Section \ref{sec:redundancy-ac-code}.
The pointwise redundancy   is upper bounded in the following way:
\begin{eqnarray*}
\lefteqn{  -\log Q^n (X_{1:n}) + \log \PROB^n (X_{1:n}) }\\
&\leq &\underbrace{\ell(C_E)}_{\textsc{(i)}} + \underbrace{\ell(C_M) + \log \PROB^n (\widetilde{X}_{1:n})}_{\textsc{(ii)}}  \, .
\end{eqnarray*}
Proposition \ref{prop:elias:encoding:length} asserts that \textsc{(i)} is negligible with respect to $R^+(\Lambda_f^n)$
and Proposition~\ref{prop:adpat-mixt-coding} asserts that the expected value of \textsc{(ii)} 
is equivalent to $R^+(\Lambda_f^n)$.

\section{Minimax redundancies}

\label{sec:minimax-redundancies}
\subsection{The Haussler-Opper lower bound}
The minimax redundancy of source classes defined by  envelopes $f$ with finite and non-decreasing hazard rate is characterized using Theorem 5 from \citep{MR1604481}. This theorem relates 
the minimax redundancy (the minimax cumulative entropy risk in the 
language of \citeauthor{MR1604481})
to the metric entropy of the class of marginal distributions with respect to Hellinger distance.
Recall that the Hellinger distance (denoted by $d_H$) between two probability
distributions $P_1$  and $P_2$ on  $\mathbb{N}$, is defined as the
$\ell_2$ distance between the square roots of the corresponding 
probability mass functions $p_1$  and $p_2$:
\begin{displaymath}
d_H^2(P_1,P_2) =
   \sum_{k\in \mathbb{N}} \Big(p_1(k)^{1/2} -p_2(k)^{1/2} \Big)^2 \, . 
\end{displaymath} 

The next lower bound on minimax redundancy can be extracted 
 from  Theorem 5  in \citep{MR1604481}. It relies on the fact that 
the Bayes redundancy is never larger than the minimax redundancy. 
\begin{thm}\label{thm:haussler:opper:lower}
Using notation and conventions from Section \ref{sec:introduction}, 
  for any prior probability distribution $\pi$ on $\Lambda_1$, 
  \begin{eqnarray*}
    R^+(\Lambda^n) 
    & \geq &  \EXP_{\pi_1}\left[ -\log \EXP_{\pi_2} e^{-n
        \frac{d^2_H(P_1,P_2)}{2}} \right] \, . 
  \end{eqnarray*}
  where $\pi_1=\pi_2= \pi$ and $P_1\sim \pi_1,P_2\sim \pi_2$ are
  picked independently.  
\end{thm}
For the sake of self-reference, a  rephrasing of the proof from
\citep{MR1604481} is given in the Appendix.

For a source class $\Lambda$,  Let ${\cal H}_{\epsilon}(\Lambda)$ 
be the $\epsilon$-entropy of $\Lambda^1$ with respect to the Hellinger metric. That is, 
 ${\cal H}_{\epsilon}(\Lambda)=\ln {\cal D}_{\epsilon}(\Lambda)$ where ${\cal D}_{\epsilon}(\Lambda)$ is the cardinality of the smallest finite partition of $\Lambda^1$ 
into sets of diameter at most $\epsilon$ when such a finite partition exists. 

The connection between the minimax redundancy of $\Lambda^n$ and 
the metric entropy of $\Lambda^1$ under Hellinger metric is a direct
consequence of Theorem \ref{thm:haussler:opper:lower}.

\begin{thm}\label{thm:haussler:opper:lower:2}
   Let  $\epsilon_n>0$ be such that 
$$\frac{n \epsilon_n^2}{8} \geq \mathcal{H}_{\epsilon_n}(\Lambda) $$ then 
  $$R^+(\Lambda^n) \geq \log (e) \mathcal{H}_{\epsilon_n}(\Lambda)
  -1 \, .$$ 
\end{thm}

\begin{proof}[Proof of Theorem \ref{thm:haussler:opper:lower:2}]
  Choose a prior $\pi$ that is uniformly distributed over an
  $\epsilon/2$-separated set of maximum cardinality.  Such a set has
  cardinality at least $\mathcal{D}_\epsilon=\mathcal{D}_\epsilon(\Lambda^1)$.

For each $P_1$ in the $\epsilon/2$-separated,
\begin{multline*}
-\log \EXP_{\pi_2} \mathe^{-n \frac{d_H^2(P_1,P_2)}{2}}
  \geq -\log \left(\frac{1}{\mathcal{D}_\epsilon}+
    \frac{\mathcal{D}_\epsilon-1}{\mathcal{D}_\epsilon}
\mathe^{-n \frac{\epsilon^2}{8}}   \right)\, . 
\end{multline*}
Averaging over $P_1$ leads to 
  \begin{align*}
    R^+(\Lambda^n) 
     &\geq -\log \left( \frac{1}{\mathcal{D}_\epsilon}
      + \mathe^{ -n\frac{\epsilon^2}{8}} \right)  \\
     &\geq \log e \sup_{\epsilon} \min \left( \mathcal{H}_{\epsilon}(\Lambda),
      \frac{n\epsilon^2}{8}\right) - 1 \, . 
  \end{align*}
\end{proof}

Up to this point, no assumption has been made regarding the behavior of $\mathcal{H}_{\epsilon}(\Lambda)$
as $\epsilon$ tends to $0$. Recall  that a measurable function $f\colon (0,\infty)\rightarrow [0,\infty)$
 is said to be 
\emph{slowly varying} at infinity if  for all  $\kappa>0$,
$
\lim_{x\rightarrow +\infty}\frac{f(\kappa x)}{f(x)}=1 
$ \citep[See][for a thorough treatment of regular
variation]{BinGolTeu89}. 

Assuming that
$x\mapsto \mathcal{H}_{1/x}(\Lambda)$
is slowly varying at infinity allows us to solve
equation $\mathcal{H}_{\epsilon}(\Lambda)=n\epsilon^2/8$  as $n$ tends
to infinity. 

Indeed, using $g(x)$ as a shorthand for
$\mathcal{H}_{1/x}(\Lambda)$, we look for $x$ satisfying 
$n/8=x^2g(x)$ or equivalently
\begin{equation}\label{eq:3}
  1 =  \frac{x}{(n/8)^{1/2}} \sqrt{g \left( (n/8)^{1/2} \frac{x}{(n/8)^{1/2}}\right)} \, . 
\end{equation}
Assume that $g$ is slowly varying. A Theorem by De
Bruijn \citep[See][Theorem 1.5.13]{BinGolTeu89} asserts that there exists slowly varying
functions $g^\#$ such that $g^\#(x) g(xg^\#(x))\sim 1$ as $x$ tends to
infinity. Moreover  all such functions are asymptotically equivalent,
they 
are called De Bruijn conjugates of $g$. If ${g}$ is slowly
varying, so is $\sqrt{g}$, and  we may  consider its De Bruijn conjugate $(\sqrt{g})^\#$.

Any sequence $(x_n)_n$  of solutions of Equation \eqref{eq:3} is such that $x_n/\sqrt{n/8}$ is
asymptotically equivalent to $\sqrt{g}^\#((n/8)^{1/2})$. 
Hence $\epsilon_n = 1/x_n \sim  (n/8)^{-1/2} (\sqrt{g}^\#((n/8)^{1/2}))^{-1}$. We may 
deduce that
\begin{eqnarray}
 R^+(\Lambda^n) &\geq &(1+o(1)) \log\mathe \, \frac{n \epsilon_n^{2}}{8} \notag \\
& = & \log \mathe
\left(\sqrt{g}^\#((n/8)^{1/2})\right)^{-2}\, .\label{eq:1}
\end{eqnarray}
The computation
of De Bruijn conjugates is usually a delicate topic
\citep[see again][]{BinGolTeu89}, but  strengthening the slow variation assumption 
simplifies the matter. We have
not been able to find a name for the next notion in the literature, although it
appears in early work by \citet{bojanic1971slowly}.  We nickname it  \emph{very slow variation} in the paper. 

\begin{dfn}
  A  continuous, non decreasing function $g\colon (0,\infty)\rightarrow [0,\infty)$ is said to be 
\emph{very slowly varying} at infinity if  for all $\eta \geq 0$ and $\kappa>0$,
$$
\lim_{x\rightarrow +\infty}\frac{g(\kappa x(g(x))^{\eta})}{g(x)}=1 \, .
$$
\end{dfn}
Note that not all slowly varying functions satisfy these
conditions. For example, $x\mapsto \exp((\ln x)^\beta)$ with
$\beta>1/2$ does not \citep[see][]{BinGolTeu89}. 

\begin{rem}
  If $g$ is very slowly varying, then for all $\alpha,\beta>0$, the
  function defined by 
  $x\mapsto (g(x^\beta))^\alpha$  is also very slowly varying. 
\end{rem}

\citeauthor{bojanic1971slowly} have proved that if $g$ is a very slowly
varying function, the De Bruijn conjugates of $g$ are asymptotically
equivalent to $1/g$ \citep[See][Corollary 2.3.4]{BinGolTeu89}. Hence, if $g$ is very slowly varying, the De
Bruijn conjugates of $x\mapsto \sqrt{g}(x)$ are asymptotically
equivalent to $1/\sqrt{g}(x).$ 

Taking advantage of the very slow variation property allows us to make
lower bound \eqref{eq:1} transparent. 

\begin{thm}\citep[Theorem 5]{MR1604481}
  \label{thm:opper:haussler}
Assume that the function over $[1,\infty)$ defined by 
$
x \mapsto {\cal H}_{1/x}(\Lambda^1) 
$
is very slowly varying,  then
$$
R^+\left(\Lambda^n\right) \geq (\log e)\mathcal{H}_{n^{-1/2}}(\Lambda)\left(1+o(1)\right) \quad\text{as $n$ tends to $+\infty$.}
$$
\end{thm}

In order to lower bound the minimax redundancy of source classes
defined by envelope distribution with non-decreasing hazard rate, in
the next section we establish that the function $t \mapsto
\int_1^{t^2} \tfrac{U(x)}{x} \mathrm{d}x$ has the very slow variation
property, then in Section  \ref{sec:minimax-rate} we check that
$\int_1^{t^2} \tfrac{U(x)}{x} \mathrm{d}x \sim
\mathcal{H}_{1/t}(\Lambda)$. 
\subsection{Slow variation and consequences}
\label{sec:slow-variation}

Let us now state some analytic properties that will prove useful when checking
 that source classes  defined by envelopes with finite and non-decreasing hazard rate are indeed small. 

\begin{prop}\label{prop:analytic:pp:1}
 Let $F$ be an absolutely continuous distribution with finite and non-decreasing hazard
 rate. Let $U\colon [1,\infty) \to \mathbb{R}$ be defined by $U(t) =
 F^{-1}(1-1/t)$. Then
 \begin{itemize}
 \item[(i)] the inverse of the hazard function, that is the function
   on $]0,\infty)$ defined by  
 $t\mapsto U(\mathe^t)= F^{-1}(1-\mathe^{-t})$ is concave.
 \item[(ii)] The function 
 $U$ is slowly varying at infinity. 
\end{itemize}
\end{prop}

The proof relies on some classical results from regular variation
theory. For the sake of self-reference, we reproduce those results
here, proofs can be found in \citep{BinGolTeu89} or in
\citep{HaaFei06}.

Theorem 1.2.6 in \cite{HaaFei06} characterizes the so-called domains of
attraction of Extreme Value Theory thanks to an integral
representation of the hazard function $-\ln \overline{F}$. We
reproduce here what concerns the Gumbel domain of attraction as 
the envelope distributions we deal with belong to this domain.

\begin{thm}\label{thm:mda:hazard} A distribution function $F$ belongs
  to the Gumbel max-domain of attraction, if and only if there
  exists $t_0< F^{-1}(1)$ and $c\colon [t_0,F^{-1}(1))\to
  \mathbb{R_+}$, and a continuous function $\psi$ such that $\lim_{t \to
    F^{-1}(1)} c(t)=  c_* \in \mathbb{R}_+$ such that for $t\in
  [t_0,F^{-1}(1))$, 
  \begin{displaymath}
    -\ln \overline{F}(t) = -\ln c(t)  + \int_{t_0}^t \frac{1}{\psi(s)}
    \mathrm{d}s 
  \end{displaymath}
where  $\lim_{t \to F^{-1}(1)} \psi'(t) =0$ and $\lim_{t \to
      F^{-1}(1)} \psi(t) =0$ when $F^{-1}(1)<\infty$. 

If $F$ belongs to the Gumbel max-domain of attraction, then $t\mapsto
F^{-1}(1-1/t)$ is slowly varying at infinity. 
\end{thm}
\begin{rem}
  Under the so-called Von Mises conditions $\psi$ may be chosen as the
  reciprocal of the hazard rate.
\end{rem}
\begin{proof}[Proof of Proposition \ref{prop:analytic:pp:1}]
(i) As $t \mapsto F^{-1}(1-\mathe^{-t})$
is the inverse of the hazard function $-\ln \overline{F}$, its derivative is  
equal to the reciprocal of the hazard rate evaluated at
$F^{-1}(1-\mathe^{-t})$. As the hazard rate is assumed to be
non-decreasing,  the derivative  of $F^{-1}(1-\mathe^{-t})$ is 
non-increasing with respect to  $t$.\\  
(ii) As we consider an absolutely continuous distribution, we may and
do assume that using the notation of Theorem \ref{thm:mda:hazard},
 $c(t)$ is  constant and that $\psi$ 
is the reciprocal of the hazard rate and that it is differentiable. 
The function $\psi$ is a positive non
increasing function, thus its derivative converges to $0$ at
infinity. Hence, by Theorem \ref{thm:mda:hazard}, the smoothed envelope distribution belongs to the Gumbel max-domain of
attraction. This entails that $t \mapsto F^{-1}(1-1/t)$ is slowly
varying at infinity.
\end{proof}

The next Theorem from \citet{bojanic1971slowly} can be found in
\citep[][Theorem 2.3.3]{BinGolTeu89}. It asserts that if $g$ is slowly
varying and if for some  $x>0$, $g(tx)/g(t)$ converges sufficiently
fast towards $1$, then $g$ is also very slowly varying. 
\begin{thm}\label{thm:bojanic:seneta}
If $g$ varies slowly at infinity and for some $x\in \mathbb{R}_+$
\begin{displaymath}
  \lim_{t \to \infty} \left(\frac{g(tx)}{g(t)}-1 \right) \ln g(t) = 0 
\end{displaymath}
then $g$ is very slowly varying, 
\begin{displaymath}
  \lim_{t \to \infty} \left(\frac{g(t(g(t))^\kappa)}{g(t)}-1 \right) = 1 
\end{displaymath}
locally uniformly in $\kappa\in \mathbb{R}$.
\end{thm}
The next proposition establishes that if $F$ is an envelope
distribution with finite non-decreasing hazard rate,
$U(t)=F^{-1}(1-1/t)$, then the function $t \mapsto \int_{1}^{t}
\tfrac{U(x)}{x}\mathrm{d}x$
satisfies the condition of Theorem \ref{thm:bojanic:seneta}.
As a byproduct of the proof, we show that  $U(t)\ln U(t)$ is asymptotically negligible with respect to $\int_{1}^{t}
\tfrac{U(x)}{x}\mathrm{d}x$ as $t$ tends to infinity. Both properties are instrumental in the derivation of
the main results of this paper. First the sequence $(\int_{1}^{n}
\tfrac{U(x)}{x}\mathrm{d}x)_n$ is asymptotically equivalent to 
$\mathcal{H}_{1/n^{1/2}}(\Lambda^1_f)$ and thus to the lower bound on
the minimax redundancy rate  for $\Lambda_f$; 
second,  by Proposition \ref{prop:adpat-mixt-coding}, it is also
equivalent to the average redundancy of the mixture encoding produced
by the \textsc{ac}-code (the negligibility of $U(t)\ln U(t)$ is used
in the proof); third, the cost of the Elias encoding of
censored symbols does not grow faster than $U(n)$. 
\begin{prop}\label{prop:analytic:pp:2}
Using notation from Proposition \ref{prop:analytic:pp:1}.
\begin{enumerate}[i)]
\item
$$
\lim_{t\rightarrow +\infty} \frac{U(t)\ln U(t)}{\int_{1}^{t} \frac{U(x)}{x}\mathrm{d}x}=0.
$$
\item
The function  $\widetilde{h} \colon [1,\infty) \rightarrow \mathbb{R}$,  $\widetilde{h}(t)= \int_1^{t^2} \tfrac{U(x)}{2x} \mathrm{d}x$ is very slowly varying.
\end{enumerate}
\end{prop}


A function $g\colon \mathbb{R}_+ \to \mathbb{R}_+$ has the \emph{extended
regular variation property}, if there exists $\gamma\in \mathbb{R}$ and
a measurable function
$a\colon \mathbb{R}_+ \to \mathbb{R}_+$ such  that for all $y\geq 1$
\begin{displaymath}
  \lim_{t\to \infty} \tfrac{g(ty)-g(t)}{a(t)}  = \int_1^y x^{\gamma-1}
  \mathrm{d}x \,.
\end{displaymath}  If $\gamma=0$, $g$ has the extended slow variation
property. The function $a$ is called the auxiliary function.  We refer to \citep{BinGolTeu89} or \citep{HaaFei06} for a thorough
treatment of this notion. Basic results from regular variation theory
assert that if $g$ has the extended slow variation property, all
possible auxiliary functions 
are slowly varying and that $\lim_{t\to \infty} a(t)/g(t)=0$.

\begin{proof}[Proof of Proposition \ref{prop:analytic:pp:2}]
 To prove i), note that  by concavity of $t\mapsto U(\mathe^t)$,
\begin{eqnarray*}
  \int_1^t \frac{U(x)}{x} \mathrm{d}x & = &   \int_0^{\ln t} {U(\mathe^s)}\mathrm{d}s \\
& \geq & \frac{\ln(t)}{2} U(t)\, . 
\end{eqnarray*}
Plugging this upper bound leads to
\begin{eqnarray*}
   \frac{U(t)\ln (U(t))}{\int_{1}^{t} \frac{U(x)}{x}\mathrm{d}x} &\leq &  2 \frac{U(t)\ln (U(t))}{U(t)\ln(t) }  = 2 \frac{\ln (U(t))}{\ln(t)}
\end{eqnarray*}
which tend to $0$ as $t$ tends to infinity \citep[Again
by][Proposition B.1.9, Point 1]{HaaFei06}. 

ii) We first prove that  $t\mapsto \int_1^t \frac{U(x)}{x} \mathrm{d}x$ has the
\emph{extended slow variation property with auxiliary function}
$U(t)$, that is 
 for all $y\geq 1$, 
\begin{displaymath}
  \lim_{t\to \infty} \frac{ \int_1^{ty} \frac{U(x)}{x} \mathrm{d}x-
    \int_1^t \frac{U(x)}{x} \mathrm{d}x}{U(t)} =  \log y \, . 
\end{displaymath}
Indeed
\begin{eqnarray*}
  \int_1^{ty} \frac{U(x)}{x} \mathrm{d}x-
    \int_1^t \frac{U(x)}{x} \mathrm{d}x & = & \int_1^{y}
    \frac{U(tx)}{x} \mathrm{d}x \\
&= & U(t) \int_1^{y}
    \frac{U(tx)}{U(t)} \frac{1}{x} \mathrm{d}x
\end{eqnarray*}
Now the desired result follows from the Uniform Convergence Theorem
from regular variation theory: if $U$ is slowly varying, then $\tfrac{U(tx)}{U(t)}$  converges uniformly to $1$ for $x\in
[1,y]$ as $t$ tends to infinity. 

In order to invoke Theorem \ref{thm:bojanic:seneta} to establish  ii)
it suffices to notice that
\begin{align*}
&\lim_{t\to \infty} \left(\frac{\int_1^{ty}
      \frac{U(x)}{x} \mathrm{d}x}{\int_1^t \frac{U(x)}{x} \mathrm{d}x}
    -1\right) \ln \int_1^t \frac{U(x)}{x} \mathrm{d}x\\
  &\quad = \lim_{t\to \infty} \frac{\int_1^{ty}
    \frac{U(x)}{x} \mathrm{d}x -\int_1^t \frac{U(x)}{x} \mathrm{d}x}{U(t)}
    \frac{U(t) \ln \int_1^t \frac{U(x)}{x} \mathrm{d}x}{\int_1^t
    \frac{U(x)}{x} \mathrm{d}x}\\
  &\quad = \ln y \lim_{t\to \infty} \frac{U(t)    \ln \int_1^t \frac{U(x)}{x} \mathrm{d}x}{\int_1^t
    \frac{U(x)}{x} \mathrm{d}x} \\
  &\quad \leq \ln y \lim_{t\to \infty} \frac{U(t)    \ln \left( U(t)\ln(t)\right)}{\int_1^t
    \frac{U(x)}{x} \mathrm{d}x} \\
  &\quad \leq \ln y \left\{ \lim_{t\to \infty} \frac{U(t)    \ln U(t)}{\int_1^t
    \frac{U(x)}{x} \mathrm{d}x}+  \lim_{t\to \infty} \frac{ U(t)\ln\ln(t)}{\int_1^t
    \frac{U(x)}{x} \mathrm{d}x}\right\}\\
  &\quad \leq \ln y  \lim_{t\to \infty} \frac{\ln\ln(t)}{2 \ln t} \\
  &\quad = 0 \,
\end{align*}
where the first inequality comes from the fact that $U$ is
non-decreasing, the second inequality from i), the third inequality
from the first intermediate result in the proof of i).

\end{proof}

\subsection{Minimax rate}
\label{sec:minimax-rate}

The $\epsilon$-entropy of envelope classes defined by envelope
distributions with  finite and non-decreasing hazard rate
is related to the behavior of the quantile function of the smoothed
envelope distribution.
\begin{prop} \label{prop:entropy:hazard:rate}
  Let $f$ be an envelope function. Assume that the smoothed envelope
  distribution $F$ has finite and non-decreasing
  hazard rate, let $U(t)=F^{-1}(1-1/t)$ be the quantile of order $1-1/t$ of the
  smoothed envelope distribution, then
  \begin{displaymath}
    {\cal H}_{\epsilon}(\Lambda_f)  = (1+o_f(1))
 \int_{0}^{1/\epsilon^2} \frac{U(x)}{2x} \mathrm{d}x
  \end{displaymath}
as $\epsilon$ tends to $0$.
\end{prop}
The proof  follows  the approach of~\citep{Bon11}. 
 It is stated in the appendix. 

The next lower bound for $R^+(\Lambda_f^n)$
follows from  a direct application of Theorem \ref{thm:opper:haussler} and Proposition \ref{prop:analytic:pp:2}:

 \begin{prop}
  Using notation from Proposition \ref{prop:entropy:hazard:rate}, 
\begin{displaymath}
R^+(\Lambda_f^n)   \geq  (1+o_f(1))(\log e)
\int_1^{n} \frac{U(x)}{2x} \mathrm{d}x 
  \end{displaymath}
  as $n$ tends to $+\infty$.
\end{prop}
A concrete corollary follows easily. 
\begin{prop}\label{prop:minimax-redundancies}
The minimax redundancy of the sub-exponential envelope class with parameters $(\alpha,\beta,\gamma)$ satisfies
\begin{multline*}
\lefteqn{R^+(\Lambda^n(\alpha,\beta,\gamma))}\\
\begin{aligned}
  &\geq \frac{\alpha}{2(\alpha+1)}\beta \left(\ln(2)\right)^{1/\alpha}\left(\log n \right)^{1+1/\alpha}\left(1+o(1)\right)\\
  &\qquad \text{as $n$ tends to $+\infty$.}
\end{aligned}
\end{multline*}
\end{prop}
\begin{proof}
Indeed, if $f$ is a sub-exponential envelope function with parameters $(\alpha,\beta,\gamma)$ one has, for $t> 1$, 
\begin{equation}
\label{eq:evalU}
\beta \left( \ln  \left(\gamma t  \right)\right)^{1/\alpha} -1 \leq U(t)  \leq \beta \left( \ln  \left(\kappa\gamma t \right)\right)^{1/\alpha} -1 
\end{equation}
where   $\kappa=1/(1-\exp(-\alpha/\beta^\alpha))$.

The lower bound follows from $\overline{F}(k)\leq f(k+1) = \gamma \exp(-((k+1)/\beta)^\alpha)$ which 
entails $\overline{F}(k)\leq 1/t  \Rightarrow k+1\geq \beta (\ln(\gamma t))^{1/\alpha}.$


\end{proof}

As observed in the introduction, 
Theorem \ref{prop:upperbound} provides an  easy upper-bound on the
minimax redundancy of envelope classes,  $R^+(\Lambda_f^n)\leq 2+\log
e  +
\tfrac{U(n)}{2} \log n $. For envelope classes with non-decreasing
hazard rate, this upper-bound is (asymptotically) within a factor of
$2$ of the minimax redundancy. Indeed, 
\begin{eqnarray*}
  \int_1^n \frac{U(x)}{2x} \mathrm{d}x & = & \frac{1}{2}\int_0^{\ln n} U\left(
    \mathe^{y}\right) \mathrm{d}y \\
& \geq & \frac{1}{2} \frac{U(n) \ln n}{2} \\
& = &  \frac{U(n) \ln n}{4} \, ,  
\end{eqnarray*}
where the inequality comes from concavity and positivity of $y \mapsto U\left(
    \mathe^{y}\right)$. 
Hence, by Proposition~\ref{prop:minimax-redundancies}, for such
envelope classes
\begin{displaymath}
  R^+(\Lambda_f^n) \geq (1+o_f(1)) \frac{U(n) \log n}{4}  \, . 
\end{displaymath}
The merit of the \textsc{ac}-code is to asymptotically achieve the
minimax lower bound while processing the message online and without
knowing the precise form of the envelope.  This is established in the
next section.

\section{Redundancy of AC-code}
\label{sec:redundancy-ac-code}

The length of the \textsc{ac}-encoding of $x_{1:n}$, is the sum of the length of the Elias encoding $C_E$ of the 
sequence of differences between records $\widetilde{\mathbf{m}}$ and of the length of the mixture encoding $C_M$
of the censored sequence $\widetilde{x}_{1:n}0$. In order to establish Theorem~\ref{thm:main-result}, we first establish 
an upper bound on the average length of $C_E$ (Proposition~\ref{prop:elias:encoding:length}). 

\subsection{Maximal inequalities}
\label{sec:maximal-inequalities}

Bounds on the codeword length of Elias encoding  and on the redundancy of the mixture code 
essentially rely on bounds on the expectation of the largest symbol
$\max(X_1,\ldots,X_n)$  collected in the next propositions.  In the sequel, 
$H_n$ denotes the $n^{\text{th}}$ harmonic number 
$$\ln(n)\leq H_n = \sum_{i=1}^n \frac{1}{i}\leq \ln(n)+1\, .  $$

 \begin{prop}\label{prop:bounds:maxim:monotonehazard}
Let $Y_1,\ldots,Y_n$  be independently identically distributed according to an absolutely continuous distribution function $F$ 
with density $f=F'$ and  non-decreasing hazard rate $f/\overline{F}$. Let $b$ be the infimum  of the hazard rate. 
Let $U(t)= F^{-1}(1-1/t)$ and  $Y_{1,1}\leq \ldots \leq Y_{n,n}$  be the order statistics.
Then, 
\begin{eqnarray*}
  \mathbb{E}[Y_{n,n} ] &\leq &U(\exp(H_n))
\\
  \mathbb{E} [ Y_{n,n}^2] &\leq &\mathbb{E} [ Y_{n,n}]^2 +2/b^2\\
  \mathbb{E} [ Y_{n,n}\ln (Y_{n,n})] &\leq & (\mathbb{E} Y_{n,n})\ln (\mathbb{E} Y_{n,n})) + 2/b^2  \text{ if $Y_i\geq 1$ a.s. }
\end{eqnarray*}
 \end{prop}
 \begin{rem}
   If the hazard rate is strictly increasing, $Y_{n,n}-U(n)$ satisfies
   a law of large numbers \citep[See][]{And70}. We refer to
   \citep{boutho2012} for more results about concentration
   inequalities for order statistics.
 \end{rem}
 \begin{rem}
   The upper bound on $\mathbb{E}[Y_{n,n} ] $ may be tight. For
   example it allows to establish that the expected value of the
   maximum of $n$ independent standard Gaussian random variables is
   less than $\sqrt{2 H_n -\ln H_n -\ln \pi}$ \citep{boutho2012}.
 \end{rem}
The proof of proposition  \ref{prop:bounds:maxim:monotonehazard} relies on a quantile coupling argument and on a sequence 
of computational steps inspired by extreme value theory  \citep{HaaFei06} and concentration of measure
theory \citep{ledoux:2001}. The proof also takes advantage of the R\'enyi representation of order statistics  \citep[See][Chapter 2]{HaaFei06}.
The next theorem rephrases this classical result. 

\begin{thm}\label{prop:renyi:representation}\textsc{(r\'enyi's representation)}
  Let $(X_{1,n},\ldots,X_{n,n})$ denote the order statistics of an independent sample picked according to a distribution function $F$.
Then $(X_{1,n},\ldots,X_{n,n})$ is distributed as $(U(\exp(E_{1,n})),\ldots,U(\exp(E_{n,n})))$ where $U\colon (1,\infty)\rightarrow \mathbb{R}$
is defined by $U(t)= \inf\{ x\colon  F(x)\geq  1-1/t\}$ and $(E_{1,n},\ldots,E_{n,n})$ are the order statistics of an $n$-sample of the 
exponential distribution with scale parameter $1$.  Agreeing on $E_{0,n}=0$,  $(E_{i,n}-E_{i-1,n})_{1\leq i\leq n}$  are independent and exponentially 
distributed with scale parameter $1/(n+1-i)$.   
\end{thm}

We will also use the following general relations on moments of  maxima of independent random variables.

\begin{prop}\label{prop:phisob:max}
Let $(Y_{1,n},\ldots,Y_{n,n})$ denote the order statistics of an independent sample picked according to a common probability distribution,
then 
\begin{displaymath}
\EXP [Y_{n,n}^2] \leq (\EXP Y_{n,n})^2 +\EXP \left[ (Y_{n,n}-Y_{n-1,n})^2\right] \, , 
\end{displaymath}
and  if the random variables $(Y_i)_{i\leq n}$  are non-negative, 
\begin{displaymath}
  \EXP\left[ Y_{n,n} \ln  Y_{n,n} \right] \leq \EXP Y_{n,n} \ln  (\EXP Y_{n,n})  + \EXP \left[  \frac{(Y_{n,n}-Y_{n-1,n})^2}{Y_{n-1,n}} \right] 
\, . 
\end{displaymath}
\end{prop}

In the proof of this proposition, $\EXP^{(i)}$ denotes conditional expectation with
  respect to $Y_1,\ldots, Y_{i-1}, Y_{i+1}, \ldots, Y_n$ and for each
  $Z_i$ denotes the maximum of $Y_1,\ldots, Y_{i-1}, Y_{i+1}, \ldots,
  Y_n$, that is $Y_{n,n}$ if $Y_i<Y_{n,n}$ and $Y_{n-1,n}$ otherwise.
Order statistics are functions of independent random variables. The next theorem, 
the proof of which can be found in \citep{ledoux:2001} has proved to be 
a powerful tool when investigating the fluctuations of independent random variables
 \citep[See for example][]{efron:stein:1981,steele1986efron,massart:2003,boluma13} . 

\begin{thm}
\label{zefron-stein}
{\sc (sub-additivity of variance and entropy.)} Let $X_1,\ldots,X_n$
be independent random variables and let
 $Z$ be
a square-integrable function of $X=\left(  X_{1},\ldots,X_{n}\right)$. 
For each $1\leq i\leq n$, 
let $Z_i$ be a square-integrable function of $X^{(i)}
=\left(  X_{1},\ldots,X_{i-1},X_{i+1}, X_{n}\right).$
Then
\begin{eqnarray*}
\operatorname{Var}\left(  Z\right) &  \leq &\sum_{i=1}^n\EXP\left[
\left(  Z-\EXP^{(i)} Z  \right)^2\right] \\
& \leq &  \sum_{i=1}^n\EXP\left[
\left(  Z- Z_i  \right)^2\right] \, , 
\end{eqnarray*}
and if $Z$ and all $Z_i, 1\leq i\leq n,$ are positive,
\begin{multline*}
\mathbb{E}\left[ Z \ln(Z)\right]-\mathbb{E}Z \ln(\mathbb{E}Z)\\
\begin{aligned}
  &\leq
    \sum_{i=1}^n \mathbb{E} \left[ \mathbb{E}^{(i)} \left[ Z \ln(Z) \right]- \mathbb{E}^{(i)}Z \ln(\mathbb{E}^{(i)}Z)\right] \\
  & \leq \sum_{i=1}^n \mathbb{E} \left[  Z \ln\left(\frac{Z}{Z_i}\right) -(Z-Z_i)\right] \, . 
\end{aligned}
\end{multline*}
\end{thm}

\begin{proof}[Proof of Proposition~\ref{prop:phisob:max}]  As \begin{math} \EXP [Y_{n,n}^2] = \operatorname{Var}(Y_{n,n}) +
    (\EXP Y_{n,n})^2 \, ,
  \end{math}
  it is enough to bound $\operatorname{Var}(Y_{n,n}) $.
As $Z=Y_{n,n}$ is a function of $n$ independent random variables
  $Y_1,\ldots, Y_n$, choosing the $Z_i$ as $\max(X^{(i)})$,
$Z_i=Z$ except possibly  when $X_i=Z$, and then 
$Z_i=Y_{n-1,n}$.
The  sub-additivity property of the variance
 imply that
  \begin{eqnarray*}
    \operatorname{Var}(Y_{n,n}) 
    & \leq & \EXP \left[ (Y_{n,n}-Y_{n-1,n})^2\right] \, .
  \end{eqnarray*}
Using the sub-additivity of entropy with the convention about $Z_i$,
  \begin{multline*}
    \EXP\left[ Y_{n,n} \ln  Y_{n,n} \right] - \EXP Y_{n,n} \ln  (\EXP Y_{n,n})  \\
    \begin{aligned}
     & \leq   \EXP \left[   Y_{n,n} \ln \frac{ Y_{n,n}}{ Y_{n-1,n}}  -(Y_{n,n}-Y_{n-1,n}) \right] \\
     & \leq   \EXP \left[   Y_{n,n} \ln \left(1+ \frac{ Y_{n,n}-Y_{n-1,n}}{ Y_{n-1,n}} \right) -(Y_{n,n}-Y_{n-1,n}) \right] \\
     & \leq   \EXP \left[  \frac{(Y_{n,n}-Y_{n-1,n})^2}{Y_{n-1,n}} \right]
    \end{aligned}
  \end{multline*}
as $\ln(1+u)\leq u$ for $u>-1$.
\end{proof}

\begin{proof}[Proof of Proposition \ref{prop:bounds:maxim:monotonehazard}]
Thanks to R\'enyi's    representation of order statistics, $\mathbb{E}[Y_{n,n}] = \mathbb{E}[ U(\exp(E_{n,n}))],$  
the proof of the first statement follows from the concavity of $t\mapsto U(\exp(t))$, that is from
Proposition~\ref{prop:analytic:pp:1}. 

By the Efron-Stein inequality (See Proposition~\ref{prop:phisob:max}),
\begin{displaymath}
  \operatorname{Var}( Y_{n,n})  \leq \mathbb{E}\left[ (Y_{n,n}-Y_{n-1,n})^2\right] \, .
\end{displaymath}
Thanks again to R\'enyi's representation, $Y_{n,n}-Y_{n-1,n}$ is distributed like $U(\exp(E_{n,n}))-U(\exp(E_{n-1,n}))$. 
By concavity, this difference is upper-bounded by
\begin{multline*}
U(\exp(E_{n,n}))-U(\exp(E_{n-1,n}))\\ 
\leq \frac{\overline{F}(U(\exp(E_{n-1,n})))}{f(U(\exp(E_{n-1,n})))}(E_{n,n}-E_{n-1,n})\, .
\end{multline*}
The two factors are independent. While $\mathbb{E}[(E_{n,n}-E_{n-1,n})^2]=2$, 
$$\tfrac{\overline{F}(U(\exp(E_{n-1,n})))}{f(U(\exp(E_{n-1,n})))}
\leq \frac{1}{b}.$$

By Proposition~\ref{prop:phisob:max},  
\begin{multline*}
 \mathbb{E} [ Y_{n,n}\ln (Y_{n,n})] \\
 \begin{aligned}
  &\leq  (\mathbb{E} Y_{n,n})\ln (\mathbb{E} Y_{n,n})) + \EXP \left[ \frac{ (Y_{n,n}-Y_{n-1,n})^2}{Y_{n-1,n}}\right] \\
  &\leq  (\mathbb{E} Y_{n,n})\ln (\mathbb{E} Y_{n,n})) +\frac{2}{b^{2}} \, .
 \end{aligned}
\end{multline*}
 \end{proof}

When handling subexponential envelopes classes, Proposition~\ref{prop:bounds:maxim:monotonehazard} provides
a handy way to upper bound the  various statistics that are used to characterize the redundancy of the \textsc{ac}-code.
If the source belongs to $\Lambda(\alpha,\beta,\gamma)$, let $Y_1,\ldots, Y_n$ be identically independently  distributed
according to the probability distribution with tail function $\overline{F}(u)= 1 \wedge \sum_{k>u} f(k)$ where $f(u)= \gamma \exp(-(u/\beta)^\alpha).$
The quantile coupling argument ensures that there exists a probability space with random variables $(X'_1,\ldots,X'_n)$ distributed like 
$(X_1,\ldots,X_n)$ and random variables $(Y'_1,\ldots,Y'_n)$ distributed like 
$(Y_1,\ldots,Y_n)$  and $X'_i\leq Y'_i$ for all $i\leq n$  almost surely. 

Let $Y_{1,n}\leq \ldots \leq Y_{n,n} $  denote again the order
statistics of  a sample $Y_1,\ldots,Y_n$ from the envelope  distribution, 
then for any non-decreasing function $g$, $\EXP [g(M_n)]\leq\EXP [g(Y_{n,n})]$.  
Using (\ref{eq:evalU}) one gets the following.

\begin{prop}\label{prop:bounds:Mn}
Let $X_1,\ldots,X_n$  be independently identically distributed according to $P\in \Lambda^1(\alpha,\beta,\gamma)$,
let $M_n=\max(X_1,\ldots,X_n)$, 
then,  
  \begin{eqnarray*}
    \EXP M_n &\leq & \beta \left( \ln  \left({\kappa\gamma e n}\right)\right)^{1/\alpha}  \, .
 \\
 \EXP [ M_n \log M_n]
& \leq &
 \beta \left( \ln  \left({\kappa\gamma e n}\right)\right)^{1/\alpha} \\
 && \times
\left(\ln \beta + \frac{1}{\alpha}\ln(\ln (\kappa\gamma e n)) \right)  +2\kappa^2\\
    \EXP [M_n^2] 
&\leq & \beta^2 \left( \ln  \left({\kappa\gamma e n}\right)\right)^{2/\alpha}+2\kappa^2.
  \end{eqnarray*}
\end{prop}
It provides a simple refinement of Lemma 4 from \citep{Bon11}.

\begin{rem}
From Proposition~\ref{prop:bounds:Mn}, it follows that the number of
distinct symbols in $X_{1:n}$ grows at most logarithmically with~$n$.
A simple argument stated at the end of the Appendix shows that, if the
hazard rate is non-decreasing, the 
number of distinct symbols may grow as fast as $U(n)= F^{-1}(1-1/n)$ where
$F$ 
is the envelope distribution. 
This suggests a brute force approach to the analysis of the
redundancy of the \textsc{ac}-code based on the next two informal observations: the redundancy of
Krichevsky-Trofimov mixtures for an alphabet with size $M_n = O_P(\log n)$,
should not be larger than $\EXP \left[ \tfrac{M_n}{2}\right] \log n $;
  the cost of the Elias encoding of records is negligible
with respect to the redundancy of Krichevsky-Trofimov mixture. This leads to a simple estimate of redundancy :
$(1+o_f(1)) \tfrac{U(n)}{2}\log n $ which is always larger than $\int_{1}^n
\tfrac{U(x)}{2x} \mathrm{d}x$, but may be within a  constant of the
optimal bound. Indeed, the cautious  analysis of the \textsc{ac}-code 
pioneered in \citep{Bon11} and pursued here allows us to recover the
exact redundancy rate of the $\textsc{ac}$-code and to establish
asymptotic minimaxity. 
\end{rem}

\subsection{Elias encoding}
\label{sec:elias-encoding}

The average length of the Elias encoding for sources  from a class
with a smoothed  envelope distribution $F$
with non-decreasing hazard rate
is $O(U(n))$ (where $U(t)=F^{-1}(1-1/t)$). It does not grow as fast as the minimax redundancy and 
as far as subexponential   envelope classes are concerned, it 
contributes in  a negligible way to the total redundancy.

\begin{prop}\label{prop:elias:encoding:length}
Let $f$ be an envelope function with associated non-decreasing hazard rate. 
Then,  for all $\PROB \in \Lambda_{f} $, the expected length of the Elias encoding of the sequence 
of record increments amongst the first $n$ symbols is upper-bounded by 
  \begin{eqnarray*}
\EXP \left[  \ell(C_E) \right] 
& \leq &     (2\log(e)+\rho)  (U(\exp(H_n))+1) \, .
  \end{eqnarray*}
where  $\rho$ is a universal constant (which may be chosen as $\rho=2$). 
\\
In general if $X_1,\ldots,X_n$  are independently identically
distributed, letting $M_n=\max(X_1,\ldots,X_n)$, the following holds
\begin{eqnarray*}
  \EXP \left[  \ell(C_E) \right] 
  & \leq & 2  H_n  ( \log(2\EXP [M_n+1])+\rho)  \, .
\end{eqnarray*}
\end{prop}

For  classes defined by power law envelopes, $M_n$ grows like a power of $n$, the last upper bound   shows that the length 
of the Elias encoding of records remains polylogarithmic with respect to $n$ while the minimax redundancy grows like a power of $n$
\citep{boucheron:garivier:gassiat:2006}. However the \textsc{ac}-code is not likely to achieve the minimax redundancy over 
classes defined by power-law envelopes. 

The last statement stems from the fact that  the Elias codelength is less than a concave function of the encoded value. 
The average Elias codelength of record differences should not be larger  than the Elias codelength of the average record difference
which is the maximum divided by the number of records. 

\begin{proof}[Proof of Proposition \ref{prop:elias:encoding:length}]
The length of the Elias codewords used to encode the sequence of record differences $\widetilde{\mathbf{m}}$ is readily upper-bounded: 
\begin{eqnarray*}
  \ell(C_E) & \leq &    \sum_{i=1}^{n^0_n}  ( 2 \log \left(1+ \widetilde{m}_i- \widetilde{m}_{i-1}\right) +\rho)\\
& \leq &   \sum_{i=1}^{n^0_n}  2 \log(e) \left( \widetilde{m}_i- \widetilde{m}_{i-1}\right) +\rho n^0_n \\
& \leq &  2\log(e) M_n +\rho n_n^0\\
& \leq &  (2\log(e)+\rho) M_n\, 
\end{eqnarray*}
for some universal  constant $\rho$. The bound on the length of the Elias encoding follows from Proposition \ref{prop:bounds:Mn}.

If we were only interested in subexponential envelope classes, this would be enough. The next lines may be used to establish 
that the length of the Elias encoding remains negligible with respect to minimax redundancy for larger  envelope source classes. 

Using  the arithmetic-geometric mean inequality and $\sum_{i=1}^{n^0_n}   \left( \widetilde{m}_i- \widetilde{m}_{i-1}\right)\leq M_n$, we also have 
\begin{eqnarray*}
  \ell(C_E) & \leq &  2   \sum_{i=1}^{n^0_n}  \log \left( 1+ \widetilde{m}_i- \widetilde{m}_{i-1}\right) + n_n^0 \rho \\
& \leq & 2 n_n^0 \log(1+  M_n/n_n^0) +n_n^0\rho \, . 
\end{eqnarray*} 
The average length of $C_E$ satisfies:
\begin{multline*}
  \EXP \left[ \ell( C_E) \right] \\
  \begin{aligned}
    & \leq \EXP \left[ n_n^0 (2\log( 1+M_n/n_n^0)+\rho)\right]\\
    & \leq  2\left( \EXP \left[ n_n^0 \log( 2M_n)\right]- \EXP n_n^0 (\log \EXP n_n^0 -\rho)\right) \\
    &\qquad \text{by concavity of $x\mapsto -x \log x$}\\
    & \leq 2\left( \Big(\sum_{i=1}^n \frac{1}{i}\Big)\EXP   \log(2 (M_n+1))- \EXP n_n^0 (\log \EXP n_n^0-\rho)\right) \\
    & \leq 2  \ln(en)   ( \log(2\EXP [M_n+1])+\rho) \, . 
  \end{aligned}
\end{multline*}
The penultimate inequality comes from the following observation. 
Any integer valued random variable can be represented as the integer part of a real valued random variable with absolutely continuous 
distribution. For example, $X_1,\ldots,X_n$  may be represented as $\lfloor Y_1\rfloor, \ldots, \lfloor Y_n\rfloor$ where 
$Y_1,\ldots,Y_n$ are i.i.d. and supported by $\cup_{n\in \N_+} [n,n+1/2]$.
Any record in  $X_1,\ldots,X_n$ comes from a record in $Y_1,\ldots, Y_n$ (but the converse may not be true). Letting 
$R_n$  denote the number of records in $Y_1,\ldots, Y_n$, we have  $n_n^0 \log (M_n)\leq R_n \log (\max(Y_1,\ldots,Y_n))$.
Moreover $R_n$ and $\max(Y_1,\ldots,Y_n)$ are independent, and $R_n$ is a sum of independent Bernoulli  random 
variables with parameters $1,1/2, \ldots, 1/n$. This entails
\begin{multline*}
  \EXP [ R_n \log (\max(Y_1,\ldots,Y_n))] \\
  \begin{aligned}
    &= \EXP R_n \EXP [\log(\max(Y_1,\ldots,Y_n))]\\
    &\leq \sum_{i=1}^n \frac{1}{i} \log(2\EXP \max(Y_1,\ldots,Y_n))\\
    &\leq H_n \log (2(\EXP M_n +1)) \, .
  \end{aligned}
\end{multline*}
\end{proof}

\subsection{Adaptive mixture coding}
\label{sec:adpat-mixt-coding}


The next proposition compares the length of the mixture encoding $C_M$ with the ideal codeword length of $\widetilde{X}_{1:n}$.

\begin{prop}
\label{prop:adpat-mixt-coding}
Let $f\colon \mathbb{N}_+\rightarrow[0,1]$ be an envelope with finite and non-decreasing hazard rate.
The expected difference between the mixture encoding of the censored sequence $\widetilde{X}_{1:n}$
and its ideal  codeword length is upper-bounded by 
\begin{eqnarray*}
  {  \EXP \left[\ell(C_M) + \log \PROB (\widetilde{X}_{1:n}) \right]  }
&\leq &  \log(e)
\int_1^n \frac{U(x)}{2x} \mathrm{d}x 
\, (1+o(1))
\end{eqnarray*}
 as $n$ tends to infinity.
\end{prop}
  
The proof of Proposition \ref{prop:adpat-mixt-coding} is organized 
in two steps. The first step consists in establishing a pointwise 
upper bound on the difference between the ideal codeword length and codeword 
length of the \textsc{ac}-code (Proposition \ref{prop:pointwise} below). This upper-bound 
consists of three summands. 
The expected value of the three summands is then upper-bounded 
under the assumption that the source  belongs to an envelope class with non-decreasing hazard rate.

\begin{prop}\textsc{(pointwise bound)}\label{prop:pointwise}
Let $i_0$ be the random integer defined by: $i_0 =1\vee \lfloor {M_n}/4\rfloor$, 
then, 
  \begin{multline*}
   -\ln Q^n(\widetilde{X}_{1:n})+\ln \PROB^n(\widetilde{X}_{1:n})\\
\leq
\underbrace{\frac{M_n (\ln(M_n)+10)}{2} +\frac{\ln n }{2} }_{\textsc{(a.i)}}
+ \underbrace{\sum_{i=i_0}^{n-1} \left( \frac{M_i}{2i+1}\right)}_{\textsc{(a.ii)}} 
  \end{multline*}
\end{prop}

\begin{proof}
  Let $C_M$ be the mixture encoding of $\widetilde{X}_{1:n},$ then $ \ell (C_M) = -\log Q^n(\widetilde{X}_{1:n})$. 
The pointwise redundancy can be decomposed into 
\begin{multline*}
 -\ln Q^n(\widetilde{X}_{1:n})+\ln \PROB^n(\widetilde{X}_{1:n})\\
 \begin{aligned}
  & = \underbrace{- \ln \KT_{M_n+1}(\widetilde{X}_{1:n}) +\ln \PROB^n(\widetilde{X}_{1:n})}_{\textsc{(A)}} \\
  & \underbrace{-\ln Q^n(\widetilde{X}_{1:n}) + \ln \KT_{M_n+1}(\widetilde{X}_{1:n})}_{\textsc{(B)}}
 \end{aligned}
\end{multline*}
where $ \KT_{M_n+1}$ is the Krichevsky-Trofimov mixture coding probability over an alphabet of cardinality $M_{n}+1$.
Summand \textsc{(a)} may be upper bounded thanks to the next bound the proof of which can be found in 
 \citep*{boucheron:garivier:gassiat:2006},
\begin{eqnarray*}
\textsc{(a)} &= &{- \ln (\KT_{M_n+1}(\widetilde{X}_{1:n}) )+\ln (\PROB^n(\widetilde{X}_{1:n})}  ) \\
& \leq&  \frac{M_n+1}{2} \ln( n) + 2 \ln(2)\, .
\end{eqnarray*}
The second summand $\textsc{(B)}$ is negative, this is 
the codelength  the \textsc{ac}-code pockets  by progressively enlarging the alphabet rather 
than using $\{0, \ldots, M_n\}$  as the alphabet. 
\citet[][in the proof of Proposition 4]{Bon11} points out a simple and useful connexion between the coding lengths 
under $Q^n$ and $\KT_{M_n+1}$,
\begin{eqnarray*}
 \textsc{(b)}&= &  -\ln Q^n(\widetilde{X}_{1:n}) + \ln \KT_{M_n+1}(\widetilde{X}_{1:n}) \\
 & = & - \sum_{i=1}^{n-1} \ln \left( \frac{2i+1+M_n}{2i+1+M_i}\right) \, . 
\end{eqnarray*}
The difference between the codelengths can be further upper bounded.
\begin{eqnarray*}
  - \sum_{i=1}^{n-1} \ln \left( \frac{2i+1+M_n}{2i+1+M_i}\right) 
    & = & - \sum_{i=1}^{n-1} \ln \left( 1+ \frac{M_n -M_i}{2i+1+M_i}\right) \\
    &\leq & - \sum_{i=i_0}^{n-1} \left( \frac{M_n -M_i}{2i+1+M_i}\right) \\
    & + & \frac{1}{2} \sum_{i=i_0}^{n-1} \left( \frac{M_n -M_i}{2i+1+M_i}\right)^2 \\
    &&
    \text{as $\ln (1+x) \geq x-x^2/2$  for $ x\geq 0$} \\
    & = & 
    \underbrace{\sum_{i=i_0}^{n-1} \left( \frac{-M_n}{2i+1+M_i}\right)}_{\textsc{(b.i)}} \\
    & + & 
    \underbrace{\sum_{i=i_0}^{n-1} \left( \frac{M_i}{2i+1+M_i}\right)}_{\textsc{(b.ii)}} \\
    & + &
    \underbrace{\frac{1}{2} \sum_{i=i_0}^{n-1} \left( \frac{M_n -M_i}{2i+1+M_i}\right)^2}_{\textsc{(b.iii)}} \, . 
\end{eqnarray*}
The upper bound on \textsc{(a)} can be used to build an upper bound on $\textsc{(a)+(b.i)}$.
\begin{align*}
  &{\textsc{(a)}}+{\textsc{(b.i)}}\\
  &\quad \leq
    {M_n} \left( \frac{ \ln( n)}{2} -\sum_{i=i_0}^{n-1} \ \frac{1}{2i+1+M_i}\right)+\frac{\ln n }{2}\\
  &\quad = M_n \Bigg( \sum_{i=i_0}^{n-1} \left( \frac{1}{2i} -\frac{1}{2i+1+M_i} \right)\\
  &\quad \phantom{\frac{M_n}{\ln(2)} \Bigg( } + \frac{\ln( n)}{2} -\sum_{i=i_0}^{n-1} \frac{1}{2i}\Bigg) +\frac{\ln n }{2}\\
  &\quad \leq {M_n}\left( \sum_{i=i_0}^{n-1} \frac{M_i+1}{(2i+1+M_i)(2i)} + \frac{H_{i_0}}{2}+ \frac{1}{2n} \right) +\frac{\ln n}{2}\\
  &\quad \leq {M_n} \sum_{i=i_0}^{n-1} \frac{M_i+1}{(2i+1)(2i)} + \frac{M_n (\ln(M_n)+2)}{2} +\frac{\ln n }{2}\, . 
\end{align*}
Adding \textsc{(b.iii)} to the first summand in the last expression, 
\begin{multline*}
  {M_n} \sum_{i=i_0}^{n-1} \frac{M_i+1}{(2i+1)(2i)} + \textsc{(b.iii)} \\
  \begin{aligned}
    & \leq {M_n} \sum_{i=i_0}^{n-1} \frac{M_i}{(2i+1)^2(2i)} + {M_n} \sum_{i=i_0}^{n-1} \frac{1}{(2i+1)(2i)} \\
    & + \frac{1}{2}\sum_{i=i_0}^{n-1} \frac{M_n^2+M_i^2}{(2i+1)^2} \\
    & \leq M_n^2 \sum_{i\geq i_0}\left( \frac{1}{2i(2i+1)^2}+\frac{1}{(2i+1)^2}\right) +  \frac{M_n}{2i_0} \\
    & \leq M_n \left( \frac{M_n}{2i_0}+\frac{1}{2i_0}\right)\\ &\leq 4  M_n   \, .
  \end{aligned}
\end{multline*}
\end{proof}

\begin{proof}[Proof of Proposition~\ref{prop:adpat-mixt-coding}]
The average redundancy of the mixture code is thus upper bounded by
\begin{displaymath}
\log(e)\Bigg(  \mathbb{E} \Big[\underbrace{\frac{M_n (\ln(M_n)+10)}{2} +\frac{\ln n }{2} }_{\textsc{(a.i)}}\Big] + 
\mathbb{E} \Big[\underbrace{\sum_{i=i_0}^{n-1} \left( \frac{M_i}{2i+1}\right)}_{\textsc{(a.ii)}} \Big]\Bigg)
\end{displaymath}
We may now use the maximal inequalities  from Proposition~\ref{prop:bounds:maxim:monotonehazard}. 
\begin{eqnarray*}
  \sum_{i=1}^{n-1}  \frac{\EXP M_i}{2i+1}
 & \leq &    \sum_{i=1}^{n-1}  \frac{ U(\exp(H_i))+1}{2i+1}     \\
 & \leq &    \sum_{i=1}^{n-1}  \frac{ U(e i)+1}{2i+1} \\
 & \leq &   \int_1^n   \frac{ U(e x)}{2x}  \mathrm{d}x + \frac{U(e)}{3} + \frac{\ln (n)}{2} 
\, .
\end{eqnarray*}
Meanwhile, letting $b$ be the infimum of the hazard rate of the envelope, 
\begin{eqnarray*}
 \lefteqn{  \mathbb{E} \Big[{\frac{M_n (\ln(M_n)+10)}{2} +\frac{\ln n }{2} }\Big] }\\
& \leq & \frac{(U(en)+1)(\ln(U(en)+1)+10)}{2} + \frac{2}{b^2} +\frac{\ln n }{2}. \\
\end{eqnarray*}
Now using Propositions \ref{prop:analytic:pp:1} and \ref{prop:analytic:pp:2} and the fact that $U$ tends to infinity at infinity one gets that
$$
\ln n + U(n)\ln U(n)
=o\left(\int_1^n   \frac{ U(e x)}{2x}  \mathrm{d}x\right)
$$
as $n$ tends to infinity and the result follows.

\end{proof}


\appendix 
\subsection{Haussler-Opper lower bound}
\label{sec:haussler-opper-lower}

\begin{proof}[Proof of Theorem \ref{thm:haussler:opper:lower}]
  The proof of the Haussler-Opper lower bound consists of ingenuously using 
Fubini's theorem and Jensen's inequality.

Throughout this proof, we think of $\Lambda^1$ as a measurable parameter space
denoted by $\Theta$ endowed with a probability distribution
$\pi$, $\theta,\theta^*,\tilde{\theta},\widehat{\theta}$ denote random
elements of $\Theta$ picked according to the prior $\pi$. Each
$P_\theta$ define an element of $\Lambda^n$. The model is assumed to
be dominated and for each $\theta \in \Theta$, $\mathrm{d}P_\theta^n$ denotes the density of
$P_\theta^n$ with respect to the dominating probability. 
In this paragraph $\int_{\mathcal{X}^n}  \ldots$ should be understood 
as integration with to the dominating probability. 

Recall  that the Bayes redundancy under prior $\pi$ satisfies
\begin{multline*}
\lefteqn{\EXP_{\pi_1} \left[ D\left( P^n_{\theta^*} , \EXP_{\pi_2}
    P^n_{\tilde{\theta}}\right)\right]}\\
 = 
   -  \int_\Theta \mathrm{d}\pi(\theta^*) \int_{\mathcal{X}^n} \mathrm{d}{P^n_{\theta^*}(x_{1:n})} 
\log {\int_\Theta \mathrm{d}\pi(\tilde{\theta}) \frac{\mathrm{d}P^n_{\tilde{\theta}} (x_{1:n})}{\mathrm{d}P^n_{{\theta}^*} (x_{1:n})}}
\end{multline*}
where $\theta^*$ (resp. $\tilde{\theta}$) is distributed according to
$\pi_1=\pi$ (resp. $\pi_2=\pi$). 
The next equation 
\begin{displaymath}
  \int_\Theta \mathrm{d}\pi(\theta^*) \int_{\mathcal{X}^n} \mathrm{d}{P^n_{\theta^*}(x_{1:n})} 
\left( \frac{\int_\Theta \mathrm{d}\pi(\tilde{\theta}) \frac{\mathrm{d}P^n_{\tilde{\theta}} (x_{1:n})}{\mathrm{d}P^n_{{\theta}^*} (x_{1:n})}}{\int_\Theta \mathrm{d}\pi(\widehat{\theta})\sqrt{\frac{\mathrm{d}P^n_{\widehat{\theta}} (x_{1:n})}{\mathrm{d}P^n_{{\theta}^*} (x_{1:n})}}} \right)= 1
\end{displaymath}
is established by repeated invokation of Fubini's Theorem as follows:
\begin{align*}
 &\int_\Theta \mathrm{d}\pi(\theta^*) \int_{\mathcal{X}^n} \mathrm{d}{P^n_{\theta^*}(x_{1:n})} 
 \left( \frac{\int_\Theta \mathrm{d}\pi(\tilde{\theta})
  \frac{\mathrm{d}P^n_{\tilde{\theta}}
    (x_{1:n})}{\mathrm{d}P^n_{{\theta}^*} (x_{1:n})}}{\int_\Theta
  \mathrm{d}\pi(\widehat{\theta})\sqrt{\frac{\mathrm{d}P^n_{\widehat{\theta}}
    (x_{1:n})}{\mathrm{d}P^n_{{\theta}^*} (x_{1:n})}}} \right)\\
  & = \int_\Theta \mathrm{d}\pi(\theta^*) \int_{\mathcal{X}^n} \sqrt{\mathrm{d}{P^n_{\theta^*}(x_{1:n})} }
  \left( \frac{\int_\Theta \mathrm{d}\pi(\tilde{\theta})
    {\mathrm{d}P^n_{\tilde{\theta}}
    (x_{1:n})}}{\int_\Theta
      \mathrm{d}\pi(\widehat{\theta})\sqrt{{\mathrm{d}P^n_{\widehat{\theta}}
      (x_{1:n})}}} \right)\\
  &= \int_{\mathcal{X}^n} \int_\Theta \mathrm{d}\pi(\theta^*)\sqrt{\mathrm{d}{P^n_{\theta^*}(x_{1:n})} }
  \left( \frac{\int_\Theta \mathrm{d}\pi(\tilde{\theta})
    {\mathrm{d}P^n_{\tilde{\theta}}
    (x_{1:n})}}{\int_\Theta
    \mathrm{d}\pi(\widehat{\theta})\sqrt{{\mathrm{d}P^n_{\widehat{\theta}}
    (x_{1:n})}}} \right)\\ 
  & = \int_{\mathcal{X}^n} 
  \left( {\int_\Theta \mathrm{d}\pi(\tilde{\theta})
    {\mathrm{d}P^n_{\tilde{\theta}}
    (x_{1:n})}} \right)\\
  & = {\int_\Theta \mathrm{d}\pi(\tilde{\theta}) \int_{\mathcal{X}^n}
    {\mathrm{d}P^n_{\tilde{\theta}}
    (x_{1:n})}} \\
  &= 1 \, .
\end{align*}
Starting from the previous equation, 
using twice  the Jensen inequality and the convexity of $-\log$ leads to
\begin{align*}
  & -  \int_\Theta \mathrm{d}\pi(\theta^*) \int_{\mathcal{X}^n} \mathrm{d}{P^n_{\theta^*}(x_{1:n})} 
  \log {\int_\Theta \mathrm{d}\pi(\tilde{\theta}) \frac{\mathrm{d}P^n_{\tilde{\theta}} (x_{1:n})}{\mathrm{d}P^n_{{\theta}^*} (x_{1:n})}} \\
  & \geq
    -  \int_\Theta \mathrm{d}\pi(\theta^*) \int_{\mathcal{X}^n} \mathrm{d}{P^n_{\theta^*}(x_{1:n})} 
    \log  {\int_\Theta \mathrm{d}\pi(\widehat{\theta})\sqrt{\frac{\mathrm{d}P^n_{\widehat{\theta}} (x_{1:n})}{\mathrm{d}P^n_{{\theta}^*} (x_{1:n})}}} 
    \\
  & \geq
    -  \int_\Theta \mathrm{d}\pi(\theta^*) \log   \int_{\mathcal{X}^n} \mathrm{d}{P^n_{\theta^*}(x_{1:n})} 
    {\int_\Theta
      \mathrm{d}\pi(\widehat{\theta})\sqrt{\frac{\mathrm{d}P^n_{\widehat{\theta}}
      (X_{1:n})}
    {\mathrm{d}P^n_{{\theta}^*} (X_{1:n})}}} \, .
\end{align*}
In the sequel, $\alpha_H(P_{\widehat{\theta}},P_{{\theta}^*})$
is a shorthand for the Hellinger affinity between
$P_{\widehat{\theta}}$ and $P_{{\theta}^*}$, recall that 
\begin{eqnarray*}
\alpha_H(P_{\widehat{\theta}},P_{{\theta}^*})^n &=& \alpha_H(P^n_{\widehat{\theta}},P^n_{{\theta}^*})\\
&=&\int_{\mathcal{X}^n} \sqrt{{\mathrm{d}P^n_{\widehat{\theta}}
    (x_{1:n})}{\mathrm{d}P^n_{{\theta}^*} (x_{1:n})}}
\end{eqnarray*}
and that
\begin{eqnarray*}
\alpha_H(P_{\widehat{\theta}},P_{{\theta}^*}) &= & 1-
d^2_H(P_{\widehat{\theta}},P_{{\theta}^*}) \\
&\leq &\exp \left( -d^2_H(P_{\widehat{\theta}},P_{{\theta}^*})\right)
\, .
\end{eqnarray*}
By Fubini's Theorem again, 
\begin{align*}
  &-  \int_\Theta \mathrm{d}\pi(\theta^*) \log   \int_{\mathcal{X}^n} \mathrm{d}{P^n_{\theta^*}(x_{1:n})} 
  {\int_\Theta
    \mathrm{d}\pi(\widehat{\theta})\sqrt{\frac{\mathrm{d}P^n_{\widehat{\theta}}
	(X_{1:n})}{\mathrm{d}P^n_{{\theta}^*} (X_{1:n})}}}
  \\
  & \geq
    -  \int_\Theta \mathrm{d}\pi(\theta^*) \log    
    {\int_\Theta \mathrm{d}\pi(\widehat{\theta}) \int_{\mathcal{X}^n} \sqrt{{\mathrm{d}P^n_{\widehat{\theta}} (X_{1:n})}{\mathrm{d}P^n_{{\theta}^*} (X_{1:n})}}} \\
  & = -  \int_\Theta \mathrm{d}\pi(\theta^*) \log    
    {\int_\Theta \mathrm{d}\pi(\widehat{\theta}) \alpha_H(P_{\widehat{\theta}},P_{{\theta}^*})^n}\\
  & \geq -  \int_\Theta \mathrm{d}\pi(\theta^*) \log    
    {\int_\Theta \mathrm{d}\pi(\widehat{\theta}) \exp\left( -n
      \frac{d^2_H(P_{\widehat{\theta}},P_{{\theta}^*})}{2}\right)}\, . 
\end{align*}
The right hand side can written as
\begin{displaymath}
  \EXP_{\pi_1} \left[ -\log \EXP_{\pi_2} \exp\left( -n
    \frac{d^2_H(P_1,P_2)}{2}\right)\right]
\end{displaymath}
where $P_1$ (resp. $P_2$) is distributed according to $\pi_1=\pi$
(resp. $\pi_2=\pi$). The proof is terminated by recalling that, 
whatever the prior,  the
minimax redundancy is not smaller than the Bayes redundancy. 
\end{proof}


\subsection{Proof of Proposition \ref{prop:entropy:hazard:rate} }
\label{sec:entropies-}
In order to alleviate notation ${\cal H}_{\epsilon} $ is used as a shorthand for $ {\cal H}_{\epsilon}(\Lambda^1_f)$.  
Upper and lower bounds for $\mathcal{H}_\epsilon $ follow by adapting the  ``flat concentration argument''
  in \cite{Bon11}. The cardinality $\mathcal{D}_\epsilon$ of the smallest partition of $\Lambda^1_f$
into subsets of diameter less than $\epsilon$ is not larger than 
the smallest cardinality of a covering by Hellinger balls of radius smaller than $\epsilon/2$. Recall that  $\Lambda^1_f$
endowed with the Hellinger distance may be considered as a subset of $\ell_2^{\mathbb{N}_+}$:
\begin{eqnarray*}
C &=&   \Big\{ (x_i)_{i>0} \colon \sum_{i>0} x_i^2= 1\Big\}\\
&&
\bigcap  \Big\{ (x_i)_{i>0} \colon \forall {i>0},  0\leq x_i \leq \sqrt{f(i)}\Big\} \, . 
\end{eqnarray*}
 Let $N_\epsilon= U(\tfrac{16}{\epsilon^2})$ ($N_\epsilon$ is 
the $1-\epsilon^2/16$ quantile of the smoothed envelop distribution). 
Let $D$ be the projection  of $C$ on the subspace generated by the first $N_\epsilon$ vectors of the 
canonical basis. Any element of $C$ is at distance at most $\epsilon/4$ of 
$D$. Any $\epsilon/4$-cover for $D$ is an $\epsilon/2$-cover for $C.$
Now $D$ is included in the intersection of  the unit ball of a $N_\epsilon$-dimensional Euclidian 
space and of an hyper-rectangle $\prod_{i=1}^{N_\epsilon} [0,\sqrt{f(k)}]$. 
An $\epsilon/4$-cover for $D$ can be extracted from any maximal $\epsilon/4$-packing of points from $D$.
From such a maximal packing, a collection of pairwise disjoint balls of radius $\epsilon/8$ can be extracted that 
fits into  $\epsilon/8$-blowup of $D$.
Let $B_m$ be the $m$-dimensional Euclidean unit ball ($\text{Vol}(B_m)=  \Gamma(1/2)^m/\Gamma(m/2+1)$
with $\Gamma(1/2)= \sqrt{\pi}$).   
By   volume comparison,  
\begin{displaymath}
\mathcal{D}_\epsilon \times  (\epsilon/8)^{N(\epsilon)} \text{Vol}(B_{N_\epsilon})\leq \prod_{i=1}^{N_\epsilon}   \Big(\sqrt{f(k)}+\epsilon/4\Big) \, , 
\end{displaymath} 
or 
$$
 {\cal H}_{\epsilon}\leq \sum_{k=1}^{N_\epsilon} \ln  \Big(\sqrt{f(k)}+\epsilon/4\Big) -\ln \text{Vol}(B_{N_\epsilon}) + N_\epsilon \ln \frac{8}{\epsilon}
$$
Let $l= U(1)$ ($l=l_f +1$). For $k\geq l,$ $f(k)= \overline{F}(k-1)(1-\overline{F}(k)/\overline{F}(k-1))$.
As the hazard rate of the envelope distribution is assumed to be non-decreasing, denoting the essential infimum  of the hazard 
rate by $b$, 
$\overline{F}(k-1) (1-e^{-b})\leq f(k)\leq \overline{F}(k-1)$. Hence, for $l\leq k\leq N_\epsilon$, $\sqrt{f(k)}\geq  \epsilon/4 \sqrt{1-e^{-b}}.$ Thus
\begin{align}
 {\cal H}_{\epsilon} &\leq \sum_{k=1}^{l_f} \ln  \Big(\sqrt{f(k)}+\epsilon/4\Big)+\sum_{k=l}^{N_\epsilon} \ln (\sqrt{f(k)})\notag \\
  &\quad -\ln \text{Vol}(B_{N_\epsilon})+ \frac{N_\epsilon -l_f}{\sqrt{1-e^{-b}}} + N_\epsilon \ln \frac{8}{\epsilon} \notag \\
  &\begin{aligned}
    &\leq  \sum_{k=l}^{N_\epsilon} \frac{1}{2}\ln \left(\frac{64\overline{F}(k-1)}{\epsilon^2}\right) -\ln \text{Vol}(B_{N_\epsilon}) \\
    &\quad + \frac{N_\epsilon -l_f}{\sqrt{1-e^{-b}}} + l_f \ln \frac{8}{\epsilon} + \sum_{k=1}^{l_f} \ln  \Big(\sqrt{f(k)}+\epsilon/4\Big).
   \end{aligned}\label{eq:entropy-upper}
\end{align}

Following \cite{Bon11}, a lower bound is derived by another volume comparison argument. 
From any partition into sets of diameter smaller than $\epsilon$, one can extract 
a  covering by balls of radius~$\epsilon$.  
Then  for any positive integer $m$,
\begin{eqnarray*}
 \mathcal{D}_{\epsilon} &\geq & \frac{\prod_{k=l}^{l_f+m} \sqrt{f(k)}}{\epsilon^m\text{Vol}(B_m)} \, . 
\end{eqnarray*}
Hence, choosing $m=N_\epsilon-l_f$
\begin{align}
   {\cal H}_{\epsilon} &\geq \sum_{k=l}^{N_\epsilon}  \ln  \sqrt{f(k)}  -\ln \text{Vol}(B_{N_\epsilon-l_f}) + (N_\epsilon-l_f) \ln \frac{1}{\epsilon} \notag \\
& \geq \sum_{k=l}^{N_\epsilon}  \frac{1}{2}\ln  \left( \frac{{\overline{F}(k-1)}(1-e^{-b})}{\epsilon^2}\right)  -\ln \text{Vol}(B_{N_\epsilon-l_f}) \label{eq:entropy-lower}
\end{align}
Now, 
\begin{eqnarray*}
\ln \text{Vol}(B_{N_\epsilon})&=&\left[ N_{\epsilon} \ln N_{\epsilon}\right] (1+o(1))\\
&=&\left[U\left(\frac{16}{\epsilon^2}\right)\ln U\left(\frac{16}{\epsilon^2}\right) \right](1+o(1))
\end{eqnarray*}
as $\epsilon$ tends to $0$. Since $N_{\epsilon}\rightarrow \infty$, we have also $\ln \text{Vol}(B_{N_\epsilon-l_f})= \left[ N_{\epsilon} \ln N_{\epsilon}\right] (1+o(1))$, as $\epsilon$ tends to $0$. \\
Now, the term 
$\sum_{k=l}^{N_\epsilon}  \frac{1}{2}\ln  \left( \frac{\overline{F}(k-1)}{\epsilon^2}\right)$
n (\ref{eq:entropy-upper}) and (\ref{eq:entropy-lower}) is treated by (\ref{eq:bizarre}). The desired result follows from the fact that $U$ and hence $U \ln (U)$ are slowly varying (Proposition \ref{prop:analytic:pp:1}) and from Proposition \ref{prop:analytic:pp:2}.



\subsection{Proof of equation (\ref{eq:bizarre})}
\label{sec:proof-equation-ref}
Performing the change of variable $y=U(x)$ ($x=1/\overline{F}(y), \tfrac{\mathrm{d}x}{\mathrm{d}y}= \tfrac{F'(y)}{(\overline{F}(y))^2}$),
\begin{eqnarray*}
   \int_1^t \frac{U(x)}{2x} \mathrm{d} x  & = & \int_{l_f-1}^{U(t)} \frac{y F'(y)}{2\overline{F}(y)} \mathrm{d}y \\
& = & \left[-\frac{y}{2} \ln (\overline{F}(y)) \right]^{U(t)}_{l_f-1} + \int_{l_f-1}^{U(t)} \frac{\ln(\overline{F}(y))}{2} \mathrm{d}y \\
& = & \frac{U(t)}{2} \ln (t) + \int_{0}^{U(t)} \frac{\ln(\overline{F}(y))}{2} \mathrm{d}y \\
& = & \int_{0}^{U(t)} \frac{\ln(t \overline{F}(x))}{2} \mathrm{d}x ,
\end{eqnarray*}
where the second equation follows by integration by parts. 

\subsection{The number of distinct symbols in a geometric sample}
\label{sec:numb-dist-symb}

The number of distinct symbols in a geometric sample has been
thoroughly investigated using analytic techniques \citep[See][and
related papers]{ArKnPr06}. The next elementary
argument from \citep{BHA13} shows that the number of distinct symbols in a geometric
sample  is on average of the same order of magnitude as the maximum. 
This is in sharp constrast with what is known for samples from
power-law distributions, See  \citep{GneHanPit07} for a general
survey, \citep{ohannessian2012rare} and \citep{MR1765619} for statistical applications.

Assume that $X_1,\ldots,X_n$ are independently, identically
distributed according to a geometric distribution with parameter $q\in
(0,1)$, that is $\PROB\{ X_1>k\}=(1-q)^k$ for $k\in \mathbb{N}$. Let 
$K_n$ denote the number of distinct symbols in $X_1,\ldots,X_n$ and
$M_n=\max(X_1,\ldots,X_n)$.
\begin{displaymath}
\EXP M_n \geq  \EXP K_n \geq \EXP M_n  - \frac{1-q}{q^2} \, . 
\end{displaymath}
\begin{proof}
Let $X_{1,n}\leq X_{2,n}\leq \ldots\leq X_{n,n}$ denote the
non-decreasing rearrangement of $X_1,\ldots,X_n$. We agree on
$X_{0,n}=0$. The difference 
$M_n-K_n$ is upper-bounded by the sum of the gaps in the non-decreasing
rearrangement:
\begin{displaymath}
  M_n -K_n \leq \sum_{k=1}^{n} (X_{k,n}  - X_{k-1,n}-1)_+
      \, . 
\end{displaymath}
The expected value of $  (X_{k,n}  - X_{k-1,n}-1)_+ $ can be readily
upper-bounded using R\'enyi's representation (Theorem
\ref{prop:renyi:representation}).  
The order statistics $X_{1,n}\leq X_{2,n}\leq \ldots\leq X_{n,n}$
are distributed like $\lfloor Y_{k,n}/\ln(1/(1-q))\rfloor$ where 
$Y_{1,n}\leq Y_{2,n} \leq \ldots \leq Y_{n,n}$ are the order
statistics of a standard exponential sample. For $j\in \mathbb{N}_+$, the event 
$\lfloor Y_{k,n}/\ln(1/(1-q))\rfloor>j +\lfloor
Y_{k-1,n}/\ln(1/(1-q))\rfloor$ is included in the event 
$ Y_{k,n} >j\ln(1/(1-q)) +Y_{k-1,n}$. As $(n-k+1)(Y_{k,n}-Y_{k-1,n})$ is
exponentially distributed, the last event has probability 
$(1-q)^{j(n-k+1)}$.
\begin{align*}
  &\EXP [ M_n -K_n]\\
  &\quad \leq \sum_{k=1}^{n} \EXP [(X_{k,n}  -
      X_{k-1,n}-1)_+]\\
  &\quad \leq \sum_{k=1}^{n} \sum_{j\in \mathbb{N}} \mathbb{P}\left\{ (X_{k,n}  -
      X_{k-1,n}-1)_+ > j\right\}\\
  &\quad \leq \sum_{k=1}^{n} \sum_{j\in \mathbb{N}_+} (1-q)^{j(n-k+1)}\\
  &\quad \leq \sum_{k=1}^{n} \frac{(1-q)^{k}}{1-(1-q)^k}\\
  &\quad \leq \sum_{k=1}^{n} \frac{(1-q)^{k}}{q} \, . 
\end{align*}
\end{proof}

As $M_n$ is concentrated around $\ln n / \ln (1/(1-q))$, this simple 
argument reveals that for geometrically distributed samples, the
number of distinct symbols is close to the largest value encountered
in the sample.  

This observation can be extended to the setting 
where the sampling distribution has finite  non-decreasing hazard rate. 
\begin{prop}
  Assume that $X_1,\ldots,X_n$ are independently, identically
distributed according to a distribution with finite non-decreasing hazard
rate over $\mathbb{N}\setminus \{0\}$. Let 
$K_n$ denote the number of distinct symbols in $X_1,\ldots,X_n$ and
$M_n=\max(X_1,\ldots,X_n)$. Then there exists a constant $\kappa$ that 
may depend on the sampling distribution but not on sample size $n$
such that 
\begin{displaymath}
\EXP M_n \geq  \EXP K_n \geq \EXP M_n  - \kappa \, . 
\end{displaymath}

\end{prop}

\subsection*{Acknowledgment}
The authors wish to thank M. 
Thomas, M. 
Ohannessian and the referees for many insightful suggestions. 


\end{document}